\pdfoutput=1
\documentclass[11pt,reqno]{amsart}
\usepackage[utf8]{inputenc}
\usepackage{amsmath,amssymb,amsthm,amsfonts}
\usepackage{mathtools}
\usepackage{bm}

\usepackage[margin=2.5cm]{geometry}

\DeclareMathOperator{\supp}{supp}
\DeclareMathOperator{\Admiss}{Admiss}
\DeclareMathOperator{\Cond}{Cond}
\DeclarePairedDelimiter{\abs}{\lvert}{\rvert}
\DeclarePairedDelimiter{\norm}{\lVert}{\rVert}
\DeclarePairedDelimiterX{\pair}[2]{\langle}{\rangle}{#1, #2}
\newcommand{\R}{\mathbb{R}}
\newcommand{\Z}{\mathbb{Z}}
\newcommand{\N}{\mathbb{N}}
\newcommand{\one}{\mathbf{1}}
\newcommand{\hp}{h^{p}}
\newcommand{\hpn}{h^p(\R^n)}
\newcommand{\prange}{0 < p \le 1}
\newcommand{\Hp}{H^{p}}
\newcommand{\Hpn}{H^p(\R^n)}
\newcommand{\Lp}{L^{p}}
\newcommand{\Lpn}{L^p(\R^n)}
\newcommand{\St}{\mathcal{S}'(\R^n)}

\theoremstyle{plain}
\newtheorem{theorem}{Theorem}[section]
\newtheorem{lemma}[theorem]{Lemma}
\newtheorem{proposition}[theorem]{Proposition}

\theoremstyle{definition}

\theoremstyle{remark}

\usepackage[shortlabels]{enumitem}
\usepackage{hyperref}

\makeatletter
\g@addto@macro\normalsize{%
  \setlength{\abovedisplayskip}{2.8pt plus 4.55pt}%
  \setlength{\belowdisplayskip}{2.8pt plus 4.55pt}%
}
\makeatother

\title[Critical molecular theory and the maximal admissible class]%
{Critical molecular theory and the maximal admissible class for Goldberg-type splittings of \texorpdfstring{$\boldsymbol{\hpn}$}{hp(Rn)}, \texorpdfstring{$\boldsymbol{\prange}$}{0 < p <= 1}}

\author{M. El Aoumari}
\address{Department of Mathematics and Statistics, Concordia University,
Montréal, Québec, Canada}
\email{mohamed.elaoumari@mail.concordia.ca}
\urladdr{https://orcid.org/0009-0003-9984-312X}

\date{\today}


\begin{document}

\begin{abstract}
We extend the classical Taibleson-Weiss molecular theory to arbitrary radii, reaching the critical decay $\abs{x}^{-n/p}$, with the resulting embedding shown sharp. This molecular theory governs the largest class of functions certifying membership in $h^p(\mathbb{R}^n)$, $0 < p \le 1$, through Goldberg's convolution splitting: a critical Hölder modulus and an $\ell^p$-Dini decay majorant, sharp on the regularity, decay, and cancellation axes.
\end{abstract}

\maketitle

\section{Introduction}
\subsection{Background and goal}
Goldberg's atomic characterization of $\hpn$~\cite{goldberg1979} carries a convolution splitting: for a fixed Schwartz $\varphi$ with $\int \varphi = 1$ and vanishing moments to order $N_p$, $f \in \hpn$ if and only if $f - \varphi * f \in \Hpn$ and $\varphi * f \in \Lpn$. The question left open since Goldberg, and pursued to its end here, is which functions may witness the splitting at all: how weak may $\varphi$ be, on regularity and on decay simultaneously, for the equivalence to persist. Peloso \& Secco~\cite{pelososecco2008} extended Goldberg's splitting to $0 < p \le 1$ through local Riesz transforms, and the largest admissible class itself remained unidentified, even at $p = 1$.

We answer this in full: for $p < 1$, the splitting characterizes $\hpn$ over $\varphi \in C^\omega_\eta(\R^n) \cap \Cond_G$ exactly when the modulus $\omega$ is dominated by the critical Hölder exponent $\delta_p := n_p - \lceil n_p \rceil + 1$ and the majorant $\eta$ satisfies an $\ell^p$-Dini summability, iterated once more at integer $n_p := n(1/p - 1)$, while at $p = 1$ no regularity is required, the class being $L^\infty_\eta(\R^n) \cap \Cond_G$ under the same condition on $\eta$ (Theorem~\ref{thm:admiss-characterization}). The resulting class $\Admiss_{p}$ is maximal on all three axes at once, regularity, decay, and cancellation, each necessity witnessed by an explicit construction (Propositions~\ref{prop:linftyc-fails}--\ref{prop:cond-g-sharp}), and it reaches the critical decay $\abs{x}^{-n/p}$, at which the classical Taibleson-Weiss molecular parameter vanishes and no power majorant survives. Membership there is governed by a new molecular theory (Theorem~\ref{lem:eta-molecule}), extending the classical theory to arbitrary radii and sharp at the critical order (Proposition~\ref{prop:log-sharp}).

\subsection{Relation to existing literature}
Molecular theory on $\hp$ was initiated by Komori~\cite{komori2001} and recast by Dafni, Lau, Picon \& Vasconcelos~\cite{dafnilaupiconvasconcelos2022}, who grade the cancellation axis through approximate moment conditions, of power type and, at $p = 1$, logarithmic~\cite{dafniyue2012, dafniliflyand2019}, keeping decay of power type. The kernel conditions under which such operators map atoms to molecules were characterized by Lau \& Vasconcelos~\cite{lauvasconcelos2026}. The present molecules grade the opposite axis, arbitrary decay with exact moments, the two gradings meeting in parallel logarithmic corrections at the critical order. No characterization of the admissible class for a Goldberg-type splitting, necessary and sufficient on regularity, decay, and cancellation at once, has previously appeared in the literature, in any range of $p$.

\subsection{Conventions}
We fix $0 < p \le 1$ and the exponents:
\begin{equation}
\label{eq:exponents}
n_p := n\Bigl(\frac{1}{p}-1\Bigr), \quad N_p := \lfloor n_p \rfloor, \quad (N_p', \delta_p) := (\lceil n_p \rceil - 1,\, n_p - N_p'),
\end{equation}
$\delta_p \in (0,1]$, $N_p' + \delta_p = n_p$, and $N_p' \le N_p$ with equality unless $n_p \in \Z^+$. We use without comment $(s+t)^p \le s^p+t^p$ and $\ell^p \hookrightarrow \ell^1$.

\section{Preliminaries}
\label{sec:prelim}
\subsection{Hardy spaces, atoms and molecules}
\label{ssec:HP}
We use the local Hardy space $\hpn$, the real Hardy space $\Hpn$, their $L^q$-atoms, and the classical Taibleson-Weiss $(p,q,\varepsilon)$-molecule: by Goldberg's atomic characterization~\cite{goldberg1979}, an $L^q$ $\hp$-atom on $B$ satisfies $\supp \subseteq B$, $\norm{\cdot}_q \le \abs{B}^{1/q-1/p}$, and vanishing moments to order $N_p$ when $B$ has radius (or side) below $1$, none otherwise, and every such atom satisfies:
\begin{equation}
\label{eq:atom-l1}
\norm{a}_1 \le \abs{B}^{1-1/p},
\end{equation}
and there is $C_{\mathrm{at}} = C_{\mathrm{at}}(n,p,q)$ with $\norm{a}_{\hp,\max} \le C_{\mathrm{at}}$ for every such atom, on a ball or a cube:
\begin{equation}
\label{eq:atom-hp-bound}
\norm{a}_{\hp, \max} \;\le\; C_{\mathrm{at}} .
\end{equation}
A tempered distribution $f$ belongs to $\hpn$ if and only if $f = \sum_j \lambda_j a_j$ in $\mathcal S'(\R^n)$ with such atoms and $\sum_j \abs{\lambda_j}^p < \infty$~\cite{goldberg1979}, the atomic quasi-norm equivalent to $\norm{\cdot}_{\hp,\max}$. An $\Hp$-atom is defined identically but with moments required at every scale. Given a ball $B$, write $A_0(B) := 2B$ and $A_k(B) := 2^{k+1} B \setminus 2^k B$ for $k \ge 1$. Every $L^q$ $\Hp$-atom is a $(p,q,\varepsilon)$-molecule for every $\varepsilon > 0$, and conversely every such molecule lies in $\Hp$ with quasi-norm at most $C_{\mathrm{mol}}(n,p,q,\varepsilon)$~\cite{taibleson1980}.

\subsection{Hölder spaces, Taylor remainders and admissible functions}
\label{ssec:holder} For $k \in \N \cup \{0\}$ and $\delta \in (0, 1]$, the Hölder space $C^{k, \delta}(\R^n)$ consists of the $\varphi \in C^k(\R^n)$ whose norm:
\[ 
\norm{\varphi}_{C^{k, \delta}} := \sum_{\abs{\nu} \le k} \norm{\partial^\nu \varphi}_\infty + \sum_{\abs{\nu} = k} [\partial^\nu \varphi]_{\delta} ,
\] 
is finite, where $[u]_{\delta} := \sup_{x \ne y} \frac{\abs{u(x) - u(y)}}{\abs{x - y}^{\delta}}$. $C^{k, \delta}_c(\R^n)$ denotes its compactly supported elements. We use this scale in the strong sense (at integer total orders $C^{k, 1} \subsetneq \Lambda_{k + 1}$, the Lipschitz-Zygmund class~\cite[Chapter~V, Section~4]{stein1970}). With the exponents of~\eqref{eq:exponents}, the space $C^{N_p', \delta_p}(\R^n)$, of total order exactly $n_p$, absorbs every Hölder class of total order at least $n_p$: $C^{k, \delta} \subseteq C^{N_p', \delta_p}$ with absolute norm control whenever $k + \delta \ge n_p$ (then $k \ge N_p'$ automatically. Gradients bound the Lipschitz semi-norm when $k > N_p'$, and $[u]_{\delta_p} \le [u]_{\delta} + 2 \norm{u}_\infty$ by splitting at $\abs{x - y} = 1$). In particular, $N_p + 1$ bounded derivatives suffice. The admissible class below carries exactly this regularity, in weighted form.

The single property of the scale used throughout is the Hölder form of the Taylor remainder: for $\varphi \in C^{k, \delta}(\R^n)$, with $P_{x_0}$ its Taylor polynomial of degree $k$ at $x_0$:
\begin{equation}
\label{eq:taylor-holder}
\abs{\varphi(x) - P_{x_0}(x)}
\;\le\; C_{n, k}
   \Bigl( \max_{\abs{\nu} = k}
   [\partial^\nu \varphi]_{\delta} \Bigr)\,
   \abs{x - x_0}^{k + \delta},
\qquad x \in \R^n .
\end{equation}
Indeed, in Taylor's theorem with integral remainder at order $k - 1$, the integrand compares $\partial^\nu \varphi$ at two points of the segment $[x_0, x]$, at distance at most $\abs{x - x_0}$, so the bracket is at most $[\partial^\nu \varphi]_{\delta}\, \abs{x - x_0}^{\delta}$ (for $k = 0$, \eqref{eq:taylor-holder} is the definition of $[\varphi]_\delta$). Applied with $(k, \delta) = (N_p', \delta_p)$, the remainder gains $\abs{x - x_0}^{n_p}$, exactly a zero margin, sufficient for the per-atom estimates of Lemmas~\ref{lem:goldberg-extension} and~\ref{lem:conv-atom}.

A \emph{radial decreasing majorant} (henceforth \emph{majorant}) is a non-increasing function $\eta \colon [0, \infty) \to (0, 1]$, grading the decay. A \emph{modulus of continuity} (henceforth \emph{modulus}) is a non-decreasing function $\omega \colon (0, 1] \to (0, \infty)$ with $\omega(0^+) = 0$ and $\omega(t) / t$ non-increasing, grading the top order regularity. For $p < 1$ and a pair $(\omega, \eta)$, the space $C^{\omega}_{\eta}(\R^n)$ consists of the $\varphi \in C^{N_p'}(\R^n)$ whose derivatives are dominated by $\eta$ and whose top order increments by $\eta$ and $\omega$ jointly:
\begin{gather}
C^{\omega}_{\eta}(\R^n)
:= \bigl\{ \varphi \in C^{N_p'}(\R^n) :
   \norm{\varphi}_{\omega, \eta} < \infty \bigr\},
\label{eq:pair-class} \\[4pt]
\norm{\varphi}_{\omega, \eta}
:= \max_{\abs{\nu} \le N_p'}\, \sup_{x \in \R^n}\,
   \frac{\abs{\partial^\nu \varphi(x)}}{\eta(\abs{x})}
\;+\; \sup_{x \in \R^n}\, \frac{1}{\eta(\abs{x})}\,
   \max_{\abs{\nu} = N_p'}\,
   \sup_{0 < \abs{h} \le 1}
   \frac{\abs{\partial^\nu \varphi(x + h)
      - \partial^\nu \varphi(x)}}{\omega(\abs{h})} ,
\notag
\end{gather}
with $N_p'$ as in~\eqref{eq:exponents}. Its order zero companion is $L^\infty_{\eta}(\R^n) := \{ \varphi \text{ measurable} : \norm{\eta(\abs{\cdot})^{-1} \varphi}_\infty < \infty \}$, which carries no modulus. The \emph{admissible moduli} are those dominated by the Hölder modulus of exponent $\delta_p$, a pointwise condition:
\begin{equation}
\label{eq:modulus-condition}
\mathcal R_p :=
\Bigl\{ \omega \text{ modulus} :\ \displaystyle
   \sup_{0 < t \le 1} \frac{\omega(t)}{t^{\delta_p}}
   \;<\; \infty \Bigr\},
\end{equation}
and the \emph{admissible majorants} are those satisfying a summability condition over the scales, of $\ell^p$-Dini type, iterated once at integer $n_p$:
\begin{equation}
\label{eq:majorant-conditions}
\mathcal D_p :=
\begin{cases}
\Bigl\{ \eta \text{ majorant} :\ \displaystyle \int_1^\infty
   \bigl( \eta(s)\, s^{n/p} \bigr)^{p}\,
   \frac{ds}{s} \;<\; \infty \Bigr\},
   & n_p \notin \Z, \\[16pt]
\Bigl\{ \eta \text{ majorant} :\ \displaystyle \int_1^\infty
   \Bigl( \int_t^\infty \eta(s)\, s^{n/p}\,
   \frac{ds}{s} \Bigr)^{p}\, \frac{dt}{t}
   \;<\; \infty \Bigr\},
   & n_p \in \Z ,
\end{cases}
\end{equation}
the integer condition implying the fractional one, since the inner integral dominates $\eta(2t)\, (2t)^{n/p} \log 2$ by monotonicity. We write $E_\eta(t) := \int_t^\infty \eta(s)\, s^{n/p}\, \tfrac{ds}{s}$, the top order tail moment of the majorant.\\
The normalization and moment conditions are collected in:
\begin{equation}
\label{eq:cond-G}
\operatorname{Cond}_G
:= \Bigl\{ \varphi \text{ measurable} :
   \int_{\R^n} \varphi = 1, \;\;
   \int_{\R^n} x^\nu \varphi(x)\, dx = 0
   \;\text{ for }\; 0 < \abs{\nu} \le N_p \Bigr\},
\end{equation}
the moment conditions being vacuous when $N_p = 0$, the subscript recalling Goldberg, whose splitting these conditions certify. The class of \emph{admissible functions for Goldberg's splitting} is then:
\begin{equation}
\label{eq:admiss-p}
\underbrace{\operatorname{Admiss}_{p}}_{\substack{\text{functions} \\ \text{with admissible} \\ \text{regularity and decay}}}
\;:=\;
\begin{cases}
\displaystyle
\Bigl( \bigcup_{\eta \in \mathcal D_p}
   L^\infty_{\eta}(\R^n) \Bigr)
   \bigcap \operatorname{Cond}_G ,
   & p = 1, \\[16pt]
\displaystyle
\Bigl( \bigcup_{\substack{\omega \in \mathcal R_p \\ \eta \in \mathcal D_p}} C^{\omega}_{\eta}(\R^n) \Bigr)
   \bigcap \operatorname{Cond}_G
\;=\; \Bigl( \bigcup_{\eta \in \mathcal D_p}
   C^{\delta_p}_{\eta}(\R^n) \Bigr)
   \bigcap \operatorname{Cond}_G ,
   & p < 1 .
\end{cases}
\end{equation}
The equality in the second case is the collapse of the union over the moduli onto its extremal slice, $C^{\omega}_{\eta} \subseteq C^{\delta_p}_{\eta} := C^{t^{\delta_p}}_{\eta}$ with $\norm{\varphi}_{\delta_p, \eta} := \norm{\varphi}_{t^{\delta_p}, \eta} \le C_\omega \norm{\varphi}_{\omega, \eta}$ for $\omega \in \mathcal R_p$, the moduli being quantified because Theorem~\ref{thm:admiss-characterization} calibrates exactly that quantification, so that~\eqref{eq:admiss-p} is maximal among all modulus-majorant classes. All integrals in~\eqref{eq:cond-G} converge absolutely for the members, since monotonicity and~\eqref{eq:majorant-conditions} force $\eta(s) = o(s^{-n/p})$, giving a power margin at every order below $n_p$ and the finite $E_\eta(1)$ at the top order of the integer case. The class~\eqref{eq:admiss-p} contains the compactly supported functions with admissible regularity and the admissible Schwartz functions:
\[
\begin{cases}
\bigl( \mathcal S(\R^n) \cup L^\infty_c(\R^n) \bigr) \cap \operatorname{Cond}_G \;\subseteq\; \operatorname{Admiss}_{p} , & p = 1, \\[2pt]
\bigl( \mathcal S(\R^n) \cup C^{N_p', \delta_p}_c(\R^n) \bigr) \cap \operatorname{Cond}_G \;\subseteq\; \operatorname{Admiss}_{p} , & p < 1 .
\end{cases}
\]
Indeed, both are witnessed at the extremal modulus $t^{\delta_p}$ and a power majorant $\eta_M(s) := (1 + s)^{-M}$ with $M > n/p$, a compactly supported function being dominated by a majorant vanishing at any rate beyond its support. Membership at slower decay is genuinely broader: $\varphi(x) = c\, (1 + \abs{x}^2)^{-n/(2p)} \bigl( \log(e + \abs{x}^2) \bigr)^{-\rho}$, with $\rho > 1/p$ at fractional $n_p$ and $\rho > 1 + 1/p$ at integer $n_p$, corrected to $\operatorname{Cond}_G$ by Lemma~\ref{lem:moment-correction}, is admissible with logarithmic margin only, and is dominated by no $\eta_M$ with $M > n/p$.

\subsection{Molecules adapted to a majorant}
\label{ssec:molecules}
We introduce a molecular structure adapted to arbitrary decay: the power tail of the Taibleson-Weiss molecules is replaced by a majorant, so that the critical rate $\abs{x}^{-n/p}$, at which the parameter $\varepsilon$ would vanish, is no longer a boundary but a graded region, resolved by the summability condition~\eqref{eq:majorant-conditions}. Let $0 < p \le 1 < q \le \infty$ and let $\eta \in \mathcal D_p$. Given a ball $B = B(x_B, r_B)$ of any radius $r_B > 0$, write $s_k := 2^k r_B$ for the inner radius of the annulus $A_k(B)$, $k \ge 1$, and set $\varepsilon_0 := 1$ and $\varepsilon_k := \eta(s_k)\, s_k^{\,n/p}$ for $k \ge 1$. A measurable function $m$ is a \emph{$(p, q, \eta)$-molecule associated with $B$} if:
\begin{enumerate}[(i)]
\item $\norm{m}_{L^q(A_k(B))} \le \varepsilon_k\, \abs{2^k B}^{1/q - 1/p}$ for all $k \ge 0$,
\item $\int_{\R^n} m(x)\, x^\nu\, dx = 0$ for all $0 \leq \abs{\nu} \le N_p$, the integrals converging absolutely by \textup{(i)} and~\eqref{eq:majorant-conditions},
\item when $n_p \in \Z$ and $r_B < 1$, moreover $\abs[\big]{\int_{2B} m(x)\, (x - x_B)^\nu\, dx} \le \bigl( 1 + \log(1/r_B) \bigr)^{-1/p}$ for every $\abs{\nu} = n_p$.
\end{enumerate}
The shape is that of the classical molecule, $\varepsilon_k$ replacing $2^{-k\varepsilon}$. For power majorants $\eta(s) = (1 + s)^{-M}$, $M > n/p$, and $r_B \ge 1$, $\varepsilon_k \le 2^{-k(M - n/p)}$, recovering the classical $(p, q, \varepsilon)$-molecule of~\cite{taibleson1980} with $\varepsilon = M - n/p$. Clause \textup{(iii)}, an approximate cancellation condition of logarithmic type, is confined to the one regime the decay grading cannot otherwise reach, $r_B < 1$ at integer $n_p$, and is void for every molecule this paper produces. Without it, the embedding of Theorem~\ref{lem:eta-molecule} degrades to $(1 + \log(1/r_B))^{1/p}$, a growth attained by Proposition~\ref{prop:log-sharp}. Molecules at the critical decay $\abs{x}^{-n/p}$ itself, where the Taibleson-Weiss parameter vanishes, appear to be new at every $p$.

\section{Supporting lemmas}
\label{sec:lemmas}
\label{sec:molecules}

\begin{lemma}
\label{lem:moment-correction} Let $L \in \N \cup \{0\}$. For any prescribed values $m_\nu \in \R$, $\abs{\nu} \le L$, there exists $v \in C^\infty_c(\R^n)$ with $\supp(v) \subseteq \overline{B(0, 1)}$ and $\int_{\R^n} y^\nu\, v(y)\, dy = m_\nu$ for all $\abs{\nu} \le L$. The construction is linear in $(m_\nu)$, so every norm of $v$ is controlled by $\max_\nu \abs{m_\nu}$.
\end{lemma}

\begin{proof}
Fix $\chi \in C^\infty_c(\overline{B(0, 1)})$, $\chi \ge 0$, $\chi \not\equiv 0$, and seek $v = \sum_{\abs{\mu} \le L} c_\mu\, y^\mu \chi$. The moment conditions form the linear system $G c = m$ with $G_{\nu, \mu} = \int y^{\nu + \mu} \chi\, dy$, the Gram matrix of the monomials in $L^2(\chi\, dy)$: it is positive definite, since $c^{\mathsf T} G c = \int \bigl( \sum_\mu c_\mu y^\mu \bigr)^2 \chi\, dy > 0$ for $c \ne 0$, a non-zero polynomial vanishing only on a null set while $\chi$ is positive on an open set. Hence $c = G^{-1} m$.
\end{proof}

\begin{lemma}
\label{lem:conv-atom} Let $0 < p \le 1 < q \le \infty$, let $\eta \in \mathcal D_p$, and let $\varphi \in C^{\delta_p}_{\eta}(\R^n)$ (for $p = 1$, $\varphi \in L^\infty_{\eta}(\R^n)$), with $\norm{\varphi}_{\delta_p, \eta}$ the corresponding norm. In particular, every $\varphi \in \operatorname{Admiss}_{p}$ qualifies. Write $\eta_{1/2}(t) := \eta(t/2)$. Let $a$ be an $L^q$ $\hp$-atom associated with the ball $B = B(x_a, r_a)$. Then:
\begin{enumerate}[label=\textup{(\roman*)}, wide, leftmargin=0pt]
\item The integral $(\varphi * a)(x) = \int_B \varphi(x - y)\, a(y)\, dy$ converges absolutely, uniformly in $x \in \R^n$, and $\varphi * a$ is bounded and continuous, with $\norm{\varphi * a}_\infty \le \norm{\varphi}_{q'} \norm{a}_q$ and $\norm{\varphi}_{q'} \le C(n, p, \eta)\, \norm{\varphi}_{\delta_p, \eta}$, where $q'$ is the Hölder conjugate of $q$. For $p < 1$, moreover $\varphi * a \in C^{N_p'}(\R^n)$ with $\partial^\nu (\varphi * a) = (\partial^\nu \varphi) * a$ for $\abs{\nu} \le N_p'$. If $\supp(\varphi) \subseteq B(0, r_\varphi)$, then $\supp(\varphi * a) \subseteq B(x_a,\, r_a + r_\varphi)$.
\item If $r_a \ge 1$, then for $x \notin 2B$:
\[
  \abs{(\varphi * a)(x)}
  \;\le\; \norm{\varphi}_{\delta_p, \eta}\,
     \abs{B}^{1 - 1/p}\, \eta_{1/2}(\abs{x - x_a}) .
\]
If $r_a < 1$, then for all $x \in \R^n$:
\[
  \abs{(\varphi * a)(x)}
  \;\le\; C(n, p, \eta)\, \norm{\varphi}_{\delta_p, \eta}\,
     \eta_{1/2}(\abs{x - x_a}) .
\]
Both bounds are independent of $q$.
\item There is $C = C(n, p, \eta, \norm{\varphi}_{\delta_p, \eta})$, independent of $a$, of $B$, and of $q$, with $\norm{\varphi * a}_p \le C$.
\end{enumerate}
\end{lemma}

\begin{proof}
The majorant bound gives $\abs{\varphi(x)} \le \norm{\varphi}_{\delta_p, \eta}\, \eta(\abs{x})$ pointwise. The majorant lies in $L^{q'}$ for every $1 \le q' \le \infty$: trivially at $q' = \infty$ since $\eta \le 1$, by $\eta(s) = o(s^{-n/p})$ with the margin $q' / p > 1$ when $1 < q' < \infty$ or when $q' = 1 > p$, and at the borderline $p = 1$, $q' = 1$ by the finite top order tail moment $E_\eta(1)$ of the integer condition in~\eqref{eq:majorant-conditions}. Hence:
\begin{equation}
\label{eq:phi-lq}
  \norm{\varphi}_{q'}
  \;\le\; C(n, p, \eta)\, \norm{\varphi}_{\delta_p, \eta} .
\end{equation}
Every $L^q$ $\hp$-atom satisfies~\eqref{eq:atom-l1}, and the separation estimate~\eqref{eq:sep-majorant} holds in both its cases, off the doubled ball and, for $r_a \le 1$, globally up to the constant $\eta(1)^{-1}$.

For (i), Absolute and uniform convergence follow from $\int_B \abs{\varphi(x - y)}\, \abs{a(y)}\, dy \le \norm{\varphi}_{q'} \norm{a}_q$, finite and independent of $x$ by~\eqref{eq:phi-lq}, and this is also the $L^\infty$ bound. Continuity follows from continuity of translation in $L^1$ applied to $a$. For $p < 1$ and $\abs{\nu} \le N_p'$, the derivative $\partial^\nu \varphi$ is bounded by $\norm{\varphi}_{\delta_p, \eta}$, so differentiation under the integral sign is allowed by the Dominated Convergence Theorem, giving $\partial^\nu (\varphi * a) = (\partial^\nu \varphi) * a$, bounded and continuous by the same argument. The support statement is immediate.

For (ii), large atoms, $r_a \ge 1$, no moment conditions used: For $x \notin 2B$, by~\eqref{eq:sep-majorant} on $B$ and~\eqref{eq:atom-l1}:
\[
  \abs{(\varphi * a)(x)}
  \;\le\; \norm{a}_1
     \sup_{y \in B} \abs{\varphi(x - y)}
  \;\le\; \norm{\varphi}_{\delta_p, \eta}\,
     \abs{B}^{1 - 1/p}\, \eta_{1/2}(\abs{x - x_a}) .
\]

For (ii), small atoms, $r_a < 1$: For $p = 1$, directly, with no regularity: both cases of~\eqref{eq:sep-majorant} and $\norm{a}_1 \le 1$ give the envelope. For $p < 1$, let $P_x$ be the Taylor polynomial of degree $N_p'$ of $y \mapsto \varphi(x - y)$ at $y = x_a$. The moments of $a$, of order up to $N_p \ge N_p'$, annihilate it, so $(\varphi * a)(x) = \int_B [\varphi(x - y) - P_x(y)]\, a(y)\, dy$, and the Hölder form of the Taylor remainder~\eqref{eq:taylor-holder} at $(k, \delta) = (N_p', \delta_p)$, its top order derivatives controlled at the points $x - z$, $z \in [x_a, y] \subseteq B$, by the semi-norm of~\eqref{eq:pair-class} at the extremal modulus and both cases of~\eqref{eq:sep-majorant}:
\[
  \abs{\varphi(x - y) - P_x(y)}
  \;\le\; C\, \norm{\varphi}_{\delta_p, \eta}\,
     \eta_{1/2}(\abs{x - x_a})\, r_a^{\,n_p} .
\]
Multiplying by $\norm{a}_1 \le \abs{B}^{1 - 1/p} = \omega_n^{1 - 1/p}\, r_a^{-n_p}$ from~\eqref{eq:atom-l1}, the powers of $r_a$ cancel identically and the envelope follows. Both bounds consume $q$ nowhere.

For (iii): For small atoms the global envelope of (ii) integrates directly:
\[
  \int_{\R^n} \abs{\varphi * a}^p\, dx
  \;\le\; C^p\, \norm{\varphi}_{\delta_p, \eta}^p
     \Bigl( C_n + C_n \int_1^\infty
     \bigl( \eta(s)\, s^{n/p} \bigr)^p\,
     \frac{ds}{s} \Bigr)
  \;=\; C' ,
\]
after the substitution $s = \abs{x} / 2$ in polar coordinates, the head bounded by $\eta \le 1$. The fractional condition of~\eqref{eq:majorant-conditions} is thus consumed exactly, and is available at both parities. For large atoms, on the doubled ball, for $q < \infty$, Hölder's Inequality with exponent $q / p > 1$, Young's Inequality, and the size condition give:
\[
  \int_{2B} \abs{\varphi * a}^p
  \;\le\; \abs{2B}^{1 - p/q}
     \bigl( \norm{\varphi}_1 \norm{a}_q \bigr)^p
  \;\le\; 2^n\, \norm{\varphi}_1^p ,
\]
the powers of $\abs{B}$ cancelling, and for $q = \infty$ the same conclusion via $\norm{\varphi * a}_\infty \le \norm{\varphi}_1 \abs{B}^{-1/p}$. On the tail, by (ii) and the same substitution:
\[
  \int_{(2B)^c} \abs{\varphi * a}^p
  \;\le\; \norm{\varphi}_{\delta_p, \eta}^p\,
     \abs{B}^{p - 1}\,
     C_n \int_{r_a}^\infty
     \bigl( \eta(s)\, s^{n/p} \bigr)^p\, \frac{ds}{s}
  \;\le\; C\, \norm{\varphi}_{\delta_p, \eta}^p ,
\]
since $\abs{B}^{p - 1} = \omega_n^{p - 1}\, r_a^{-n_p p} \le 1$ for $r_a \ge 1$. Combining, and using~\eqref{eq:phi-lq} to absorb $\norm{\varphi}_1$, proves (iii).

\noindent\textit{(Only the regularity and decay of $\varphi$ enter, the former only for $p < 1$ and $r_a < 1$.)}
\end{proof}

\begin{lemma}
\label{lem:termwise-convolution} Let $0 < p \le 1 < q \le \infty$ and let $f = \sum_i \mu_i a_i$, with convergence in $\St$, where each $a_i$ is an $L^q$ $\hp$-atom on the ball $B_i = B(x_i, r_i)$ and $\sum_i \abs{\mu_i}^p < \infty$. Let $\eta \in \mathcal D_p$ and $\varphi \in C^{\delta_p}_{\eta}(\R^n)$ (for $p = 1$, $\varphi \in L^\infty_{\eta}(\R^n)$). In particular any $\varphi \in \operatorname{Admiss}_{p}$ qualifies. Then the series $h := \sum_i \mu_i\, (\varphi * a_i)$ converges absolutely and uniformly on $\R^n$, and in $\St$. Its sum is a bounded continuous function, with $\displaystyle \norm{h}_\infty \le C(n, p, \eta, \norm{\varphi}_{\delta_p, \eta})\, \bigl( \sum_i \abs{\mu_i}^p \bigr)^{1/p}$, and is independent of the chosen atomic decomposition of $f$. Accordingly, we \emph{define} $\varphi * f := h$. When $\varphi \in \mathcal S(\R^n)$, $h$ coincides with the usual convolution of $\mathcal S$ with $\St$, given by $(\varphi * f)(x) = \pair{f}{\tilde\varphi(\cdot - x)}$, $\tilde\varphi(y) := \overline{\varphi(-y)}$. The conditions of $\operatorname{Cond}_G$ play no role in this lemma.
\end{lemma}

\begin{proof}
Throughout, $f_N := \sum_{i \le N} \mu_i a_i \to f$ in $\St$, and $\sum_i \abs{\mu_i} \le \bigl( \sum_i \abs{\mu_i}^p \bigr)^{1/p}$ since $p \le 1$. Note also that $\eta$ vanishes at infinity, for otherwise $\eta(s)\, s^{n/p}$ would be bounded below by a multiple of $s^{n/p}$ and~\eqref{eq:majorant-conditions} would fail.

By Lemma~\ref{lem:conv-atom}, for large atoms ($r_i \ge 1$), we have $\norm{\varphi * a_i}_\infty \le \norm{\varphi}_\infty \norm{a_i}_1 \le \norm{\varphi}_{\delta_p, \eta}$ by~\eqref{eq:atom-l1}, and for small atoms ($r_i < 1$) the global envelope of Lemma~\ref{lem:conv-atom}\textup{(ii)} with $\eta_{1/2} \le 1$ gives $\norm{\varphi * a_i}_\infty \le C(n, p, \eta)\, \norm{\varphi}_{\delta_p, \eta}$. Hence:
\begin{equation}
\label{eq:unif-Linfty}
  \sup_i\, \norm{\varphi * a_i}_\infty
  \;\le\; C_\varphi := C(n, p, \eta)\,
     \norm{\varphi}_{\delta_p, \eta} ,
\end{equation}
and, by $\ell^p \hookrightarrow \ell^1$, the series defining $h$ converges absolutely and uniformly on $\R^n$, with $\norm{h}_\infty \le C_\varphi \sum_i \abs{\mu_i}$. Each summand is continuous by Lemma~\ref{lem:conv-atom}\textup{(i)}, so $h$ is continuous. For $\zeta \in \mathcal S \subseteq L^1$, uniform convergence gives $\int \bigl( h - \sum_{i \le N} \mu_i\, \varphi * a_i \bigr)\, \overline{\zeta} \to 0$, that is, convergence in $\St$.

Introduce a truncation: Fix $\chi \in C^\infty_c(\R^n)$ with $\chi \equiv 1$ on $B(0, 1)$ and $\supp(\chi) \subseteq B(0, 2)$, and set, for $R \ge 3$:
\[
  \varphi^{(R)} := \varphi\, \chi(\cdot / R),
  \qquad
  \psi^{(R)} := \varphi - \varphi^{(R)}
  = \varphi\, \bigl( 1 - \chi(\cdot / R) \bigr),
\]
so that $\varphi^{(R)} \in L^\infty_c(\R^n)$ and $\supp(\psi^{(R)}) \subseteq \{ \abs{x} \ge R \}$. By the Leibniz rule, the derivatives of $\chi(\cdot / R)$ carrying factors $R^{-\abs{\beta}} \le 1$, both truncates lie in $C^{\delta_p}_{\eta}$ with $\norm{\varphi^{(R)}}_{\delta_p, \eta} + \norm{\psi^{(R)}}_{\delta_p, \eta} \le C_\chi\, \norm{\varphi}_{\delta_p, \eta}$, uniformly in $R$ (for $p = 1$, the same for the $L^\infty_\eta$ norms). The tail piece is controlled by its support alone. For any $L^q$ $\hp$-atom $a$ on $B(x_a, r_a)$, the convolution $\psi^{(R)} * a$ vanishes on $\abs{x - x_a} < R - 1$, since there $\abs{x - y} < R$ for every $y \in B(x_a, r_a) \subseteq B(x_a, \max(r_a, 1))$ when $r_a < 1$, and for $r_a \ge 1$ directly $\norm{\psi^{(R)} * a}_\infty \le \norm{\psi^{(R)}}_\infty \norm{a}_1 \le \norm{\varphi}_{\delta_p, \eta}\, \eta(R)$ by~\eqref{eq:atom-l1} and the support of $\psi^{(R)}$. On $\abs{x - x_a} \ge R - 1 \ge 2$, the envelopes of Lemma~\ref{lem:conv-atom}\textup{(ii)}, applied to $\psi^{(R)}$, give:
\[\abs{(\psi^{(R)} * a)(x)} \le C(n, p, \eta)\, C_\chi\, \norm{\varphi}_{\delta_p, \eta}\, \eta_{1/2}(\abs{x - x_a}).
\]
Hence:
\begin{equation}
\label{eq:truncation-tail}
  \sup_a\, \norm{\psi^{(R)} * a}_\infty
  \;\le\; C(n, p, \eta, \chi)\, \norm{\varphi}_{\delta_p, \eta}\,
     \eta\Bigl( \frac{R - 1}{2} \Bigr)
  \;\xrightarrow[R \to \infty]{}\; 0 ,
\end{equation}
the supremum over all $L^q$ $\hp$-atoms, and the limit by the vanishing of $\eta$ at infinity.

At fixed truncation, the sum is decomposition-independent: fix $R$ and set $h_R := \sum_i \mu_i\, (\varphi^{(R)} * a_i)$, convergent absolutely and uniformly by~\eqref{eq:unif-Linfty} applied to $\varphi^{(R)}$, whose norm $\norm{\varphi^{(R)}}_{\delta_p, \eta}$ is bounded uniformly in $R$ by the truncation estimates. For $\zeta \in \mathcal S(\R^n)$, one has $\tilde\varphi^{(R)} * \zeta \in \mathcal S(\R^n)$. It is smooth, with $\partial^\nu (\tilde\varphi^{(R)} * \zeta) = \tilde\varphi^{(R)} * \partial^\nu \zeta$, and rapidly decaying, since $\supp(\tilde\varphi^{(R)}) \subseteq \overline{B(0, 2R)}$ gives $(1 + \abs{x - y})^{-L} \le (1 + 2R)^{L} (1 + \abs{x})^{-L}$ for $\abs{y} \le 2R$. By Fubini's Theorem on each finite partial sum and $\St$-convergence of $f_N \to f$ against this fixed test function:
\[
  \int_{\R^n} h_R\, \overline{\zeta}\; dx
  \;=\; \lim_{N} \sum_{i \le N} \mu_i
     \int a_i\, \overline{ (\tilde\varphi^{(R)} * \zeta)
     }\; dx
  \;=\; \pair[\big]{f}{\tilde\varphi^{(R)} * \zeta} ,
\]
the left limit by uniform convergence against $\zeta \in L^1$. The right-hand side does not involve the decomposition. Hence, for two decompositions of the same $f$, the resulting bounded continuous functions $h_R$ and $h_R'$ agree as tempered distributions, so $h_R = h_R'$.

The independence passes to the limit: By~\eqref{eq:truncation-tail}:
\[
  \norm{h - h_R}_\infty
  \;\le\; \Bigl( \sum_i \abs{\mu_i} \Bigr)\,
     \sup_i\, \norm{\psi^{(R)} * a_i}_\infty
  \;\xrightarrow[R \to \infty]{}\; 0 ,
\]
so $h = \lim_R h_R$ uniformly, and $h$ inherits the decomposition-independence of the $h_R$. The definition $\varphi * f := h$ is therefore unambiguous.

Coherence for Schwartz $\varphi$: If $\varphi \in \mathcal S(\R^n)$, testing the $\St$-convergence $f_N \to f$ against $\tilde\varphi(\cdot - x) \in \mathcal S$ gives, for every $x$, $\sum_{i \le N} \mu_i (\varphi * a_i)(x) = \pair{f_N}{\tilde\varphi(\cdot - x)} \to \pair{f}{\tilde\varphi(\cdot - x)} = (\varphi * f)(x)$. The left side converges to $h(x)$ by the uniform convergence established at the start of the proof, so $h = \varphi * f$ pointwise.
\end{proof}

\begin{lemma}
\label{lem:goldberg-extension} Let $0 < p \le 1 < q \le \infty$ and $\varphi \in \operatorname{Admiss}_{p}$, with membership witnessed by $\eta \in \mathcal D_p$ and $\norm{\varphi}_{\delta_p, \eta}$ the norm of~\eqref{eq:pair-class} at the extremal modulus (for $p = 1$, the norm of $L^\infty_\eta$). Write $\eta_{1/2}(t) := \eta(t/2)$, again a majorant in $\mathcal D_p$ with comparable integrals. Let $A$ be a large-ball $L^q$ $\hp$-atom on $B_A = B(x_A, r_A)$, $r_A \ge 1$, and $a$ a small-ball $L^q$ $\hp$-atom on $B_a = B(x_a, r_a)$, $r_a < 1$. Then, with vanishing moments up to order $N_p$ in every case:
\begin{enumerate}[label=\textup{(\roman*)}, wide, leftmargin=0pt]
\item $A_\varphi := A - \varphi * A$ is, up to a multiplicative constant depending only on $(n, p, q, \eta, \norm{\varphi}_{\delta_p, \eta})$, a $(p, q, \eta_{1/2})$-molecule of $\Hpn$ associated with $B_A$, at the integrability exponent $q$ of the input atom, which $A_\varphi$ inherits on $B_A$ and which is not improved. For the power majorants $\eta(s) = (1 + s)^{-M}$, $M > n/p$, in particular for admissible Schwartz $\varphi$, this is a classical $(p, q, \varepsilon)$-molecule, for every $\varepsilon \in (0, M - n/p)$. If moreover $\supp(\varphi) \subseteq B(0, r_\varphi)$, then $A_\varphi$ is, up to a constant depending only on $(n, p, q, \norm{\varphi}_1, r_\varphi)$, an $L^q$ $\Hp$-atom associated with $B(x_A,\, r_A + r_\varphi)$.
\item $a_\varphi := \varphi * a$ is, up to a multiplicative constant depending only on $(n, p, \eta, \norm{\varphi}_{\delta_p, \eta})$, a $(p, \tilde q, \eta_{1/2})$-molecule of $\Hpn$ associated with the translated unit ball $B(x_a, 1)$, simultaneously for every $1 < \tilde q \le \infty$. In particular, the integrability exponent $q$ of the input atom is not seen by the conclusion, and for the power majorants these are classical $(p, \tilde q, \varepsilon)$-molecules. If moreover $\supp(\varphi) \subseteq B(0, r_\varphi)$, then $a_\varphi$ is, up to a constant depending only on $(n, p, r_\varphi, \norm{\varphi}_{\delta_p, \eta})$, an $L^\infty$ $\Hp$-atom associated with $B(x_a,\, r_a + r_\varphi)$, hence an $L^{\tilde q}$ atom for every $\tilde q$.
\end{enumerate}
\end{lemma}

\begin{proof}
By the convergence discussion following~\eqref{eq:admiss-p}, $\int_{\R^n} (1 + \abs{y})^{N_p}\, \abs{\varphi(y)}\, dy \le C_\eta\, \norm{\varphi}_{\delta_p, \eta} < \infty$. In particular, $\varphi \in L^1$ with $\norm{\varphi}_1 \le C_\eta \norm{\varphi}_{\delta_p, \eta}$, all moment integrals below converge absolutely, and Fubini's Theorem applies throughout. We record the separation estimate. For $z \in B(x_0, r_0)$ and $\abs{x - x_0} \ge 2 r_0$, monotonicity and $\abs{x - z} \ge \tfrac12 \abs{x - x_0}$ give:
\begin{equation}
\label{eq:sep-majorant}
  \eta(\abs{x - z})
  \;\le\; \eta_{1/2}(\abs{x - x_0}),
\end{equation}
and when $r_0 \le 1$ the same bound holds for all $x$ up to the constant $\eta(1)^{-1}$, since on $\abs{x - x_0} \le 2 r_0 \le 2$ one has $\eta(\abs{x - z}) \le 1$ while $\eta_{1/2}(\abs{x - x_0}) \ge \eta(1)$. That $\eta_{1/2} \in \mathcal D_p$ follows from the substitution $s \mapsto 2 s$ in~\eqref{eq:majorant-conditions}. Finally, every $L^q$ $\hp$-atom satisfies~\eqref{eq:atom-l1}.

We now prove part (i), starting with the moments. By Fubini's Theorem, the substitution $z = x - y$, and the Binomial Theorem, for $\abs{\nu} \le N_p$:
\[
  \int (\varphi * A)(x)\, x^\nu\, dx
  = \sum_{\beta \le \nu} \binom{\nu}{\beta}
     \Bigl( \int A(z)\, z^\beta\, dz \Bigr)
     \Bigl( \int \varphi(y)\, y^{\nu - \beta}\, dy \Bigr) .
\]
The term $\beta = \nu$ equals $\int A\, x^\nu\, dx$ by the normalization in $\operatorname{Cond}_G$, and every term $\beta < \nu$ vanishes by the moment conditions of $\operatorname{Cond}_G$. Hence, $\int A_\varphi\, x^\nu\, dx = 0$, for all $\abs{\nu} \le N_p$. (When $N_p = 0$ only $\nu = 0$ occurs and the identity is direct.)

For the size, by the Triangle and Young Inequalities:
\begin{equation}
\label{eq:mol-size-A}
  \norm{A_\varphi}_q
  \;\le\; \bigl( 1 + \norm{\varphi}_1 \bigr)
     \norm{A}_q
  \;\le\; \bigl( 1 + C_\eta
     \norm{\varphi}_{\delta_p, \eta} \bigr)
     \abs{B_A}^{1/q - 1/p} .
\end{equation}

As for the tail, for $\abs{x - x_A} \ge 2 r_A$ we have $A_\varphi(x) = -(\varphi * A)(x)$, and by~\eqref{eq:sep-majorant} on $B_A$ together with~\eqref{eq:atom-l1}:
\[
  \abs{A_\varphi(x)}
  \;\le\; \norm{A}_1
     \sup_{y \in B_A} \abs{\varphi(x - y)}
  \;\le\; \norm{\varphi}_{\delta_p, \eta}\,
     \abs{B_A}^{1 - 1/p}\,
     \eta_{1/2}(\abs{x - x_A}) .
\]
This tail bound passes to the annuli of the definition: since $\eta_{1/2}$ is non-increasing, on $A_k(B_A)$, $k \ge 1$, of inner radius $s_k = 2^k r_A$, it gives $\norm{A_\varphi}_{L^q(A_k(B_A))} \le \norm{\varphi}_{\delta_p, \eta}\, \abs{B_A}^{1 - 1/p}\, \eta_{1/2}(s_k)\, \abs{A_k(B_A)}^{1/q}$, and since $\abs{A_k(B_A)} = (2^n - 1)\, 2^{kn} \abs{B_A}$, $\abs{2^k B_A}^{1/p} = \omega_n^{1/p}\, s_k^{\,n/p}$ and $\abs{B_A}^{1 - 1/p} = \omega_n^{1 - 1/p}\, r_A^{-n_p} \le \omega_n^{1 - 1/p}$ as $r_A \ge 1$, this reads $\norm{A_\varphi}_{L^q(A_k(B_A))} \le (2^n - 1)^{1/q}\, \omega_n\, \norm{\varphi}_{\delta_p, \eta}\, \varepsilon_k\, \abs{2^k B_A}^{1/q - 1/p}$, with $\varepsilon_k$ the sequence of $\eta_{1/2}$. Together with~\eqref{eq:mol-size-A} on $2 B_A$, this is the condition (i) of a $(p, q, \eta_{1/2})$-molecule associated with $B_A$, up to the stated constant, the moments were verified above, and the clause (iii) is void since $r_A \ge 1$. For the power majorants $\eta(s) = (1 + s)^{-M}$, $M > n/p$, the tail bound implies the classical molecular decay. Fix $\varepsilon \in (0, M - n/p)$ and set $\theta_\varepsilon := n(1/p - 1/q) + \varepsilon$, so that $(\theta_\varepsilon - M) q < -n$. Then:
\begin{equation}
\label{eq:power-to-classical}
  \Bigl( \int_{\abs{x - x_A} \ge 2 r_A}
     \abs{A_\varphi(x)}^q
     \abs{x - x_A}^{\theta_\varepsilon q}\, dx
  \Bigr)^{1/q}
  \lesssim \norm{\varphi}_{\delta_p, \eta}\,
     \abs{B_A}^{1/q - 1/p}\, r_A^{\theta_\varepsilon} ,
\end{equation}
with the usual modification at $q = \infty$, the tail bound producing the additional factor $r_A^{n - M} \le 1$, absorbed since $r_A \ge 1$ and $M > n$. Together with~\eqref{eq:mol-size-A}, this is the classical $(p, q, \varepsilon)$-molecule condition of Subsection~\ref{ssec:HP}, in its equivalent weighted integral form, the annulus form following by restricting the integral to each $A_k(B_A)$, for every $\varepsilon \in (0, M - n/p)$, in particular for admissible Schwartz $\varphi$.

In the compactly supported case, if $\supp(\varphi) \subseteq B(0, r_\varphi)$ then $\supp(A_\varphi) \subseteq B^\star := B(x_A, r_A + r_\varphi)$. By~\eqref{eq:mol-size-A} and $\abs{B_A}^{1/q - 1/p} \le (1 + r_\varphi)^{n(1/p - 1/q)} \abs{B^\star}^{1/q - 1/p}$ (using $r_A \ge 1$ and $1/p - 1/q > 0$), the size condition of an $\Hp$-atom on $B^\star$ holds up to the stated constant, and the moments are those already verified.

We now prove part (ii), starting again with the moments. The same binomial expansion gives $\int a_\varphi(x)\, x^\nu\, dx = 0$ for $\abs{\nu} \le N_p$, every factor $\int a(z)\, z^\beta\, dz$ vanishing by the moment conditions of the small atom, so no condition on $\varphi$ is used.

For the pointwise envelope, all conclusions of (ii) rest on the intermediate bound:
\begin{equation}
\label{eq:small-weighted}
  \abs{a_\varphi(x)}
  \;\le\; C(n, p, \eta)\,
     \norm{\varphi}_{\delta_p, \eta}\,
     \eta_{1/2}(\abs{x - x_a}),
  \qquad x \in \R^n .
\end{equation}
For $p = 1$, this is immediate and consumes no regularity: $\abs{a_\varphi(x)} \le \norm{a}_1 \sup_{y \in B_a} \abs{\varphi(x - y)}$, the supremum bounded through both cases of~\eqref{eq:sep-majorant} and $\norm{a}_1 \le 1$. For $p < 1$, let $P_x$ be the Taylor polynomial of degree $N_p'$ of $y \mapsto \varphi(x - y)$ at $y = x_a$. The moments of $a$, of order up to $N_p \ge N_p'$, annihilate it, so $a_\varphi(x) = \int_{B_a} [\varphi(x - y) - P_x(y)]\, a(y)\, dy$, and the Hölder form of the Taylor remainder~\eqref{eq:taylor-holder} at $(k, \delta) = (N_p', \delta_p)$, its top order derivatives controlled at the points $x - z$, $z \in [x_a, y] \subseteq B_a$, by the semi-norm of~\eqref{eq:pair-class} at the extremal modulus and both cases of~\eqref{eq:sep-majorant}, gives:
\[
  \abs{\varphi(x - y) - P_x(y)}
  \;\le\; C\, \norm{\varphi}_{\delta_p, \eta}\,
     \eta_{1/2}(\abs{x - x_a})\, r_a^{\,n_p} .
\]
Multiplying by $\norm{a}_1 \le \abs{B_a}^{1 - 1/p} = \omega_n^{1 - 1/p}\, r_a^{-n_p}$ from~\eqref{eq:atom-l1}, the powers of $r_a$ cancel identically and~\eqref{eq:small-weighted} follows.

We next show the molecule property, for every exponent. Fix $1 < \tilde q \le \infty$. Since $\eta(s) = o(s^{-n/p})$ with $\eta \le 1$, the majorant lies in $L^{\tilde q}$ for every $1 \le \tilde q \le \infty$, the borderline $p = 1$, $\tilde q = 1$ being the convergent moment integral of order zero. Hence, \eqref{eq:small-weighted} gives $\norm{a_\varphi}_{\tilde q} \le C\, \norm{\varphi}_{\delta_p, \eta}$, while $\abs{B(x_a, 1)}^{1/\tilde q - 1/p}$ is a dimensional constant, and for $\abs{x - x_a} \ge 2$ the envelope passes to the annuli of $B(x_a, 1)$ exactly as in part (i), the amplitude being now the constant $C(n, p, \eta)\, \norm{\varphi}_{\delta_p, \eta}$, giving the condition (i) of a $(p, \tilde q, \eta_{1/2})$-molecule associated with $B(x_a, 1)$ up to the constant $(2^n - 1)^{1/\tilde q}\, \omega_n^{1/p}\, C(n, p, \eta)\, \norm{\varphi}_{\delta_p, \eta}$, and the clause (iii) is void since that radius is $1$. The constants are independent of $\tilde q$ and of the exponent $q$ of the input atom, and the moments were verified above. The power-majorant case specializes through~\eqref{eq:power-to-classical}, run at the exponent $\tilde q$ and the translated unit ball, where the radius factors are absent.

In the compactly supported case, if $\supp(\varphi) \subseteq B(0, r_\varphi)$ then $\supp(a_\varphi) \subseteq B^\star := B(x_a, r_a + r_\varphi)$, and~\eqref{eq:small-weighted} gives $\norm{a_\varphi}_\infty \le C\, \norm{\varphi}_{\delta_p, \eta}$. Since $\abs{B^\star} \le \omega_n (1 + r_\varphi)^n$ (as $r_a < 1$), we get $\norm{a_\varphi}_\infty \le C(n, p, \eta, r_\varphi)\, \norm{\varphi}_{\delta_p, \eta}\, \abs{B^\star}^{-1/p}$, the size condition of an $L^\infty$ $\Hp$-atom on $B^\star$, hence of an $L^{\tilde q}$ atom for every $\tilde q$ by Hölder's Inequality on $B^\star$. The moments are those verified above.

\noindent\textit{(The moment conditions of $\operatorname{Cond}_G$ are consumed only in \textup{(i)}, the regularity of $\varphi$ only in \textup{(ii)}, and only for $p < 1$. The decay is consumed in both parts, producing the molecular tails at the halved majorant, and is trivialized when $\varphi$ is compactly supported, where the outputs are atoms.)}
\end{proof}

\section{Molecular theory}
\label{sec:molecular}
\begin{theorem}
\label{lem:eta-molecule} Let $0 < p \le 1 < q \le \infty$ and $\eta \in \mathcal D_p$. Every $(p, q, \eta)$-molecule $m$, associated with a ball $B(x_B, r_B)$ of any radius $r_B > 0$, lies in $\Hpn$, with $\norm{m}_{\Hp, \mathrm{at}} \le C(n, p, q, \eta)$, the constant independent of the molecule, of $x_B$, and of $r_B$.
\end{theorem}

\begin{proof}
Translations preserve the conditions (i)--(iii), the annulus bounds and the moment clause \textup{(iii)} being written in $x - x_B$ and the vanishing moments by the binomial identity, which uses all orders up to $N_p$ at once, as well as the quasi-norm $\norm{\cdot}_{\Hp, \mathrm{at}}$. With that, without loss of generality, we may take the center of the ball to be the origin, and we write $r := r_B$, $s_k := 2^k r$ for $k \ge 0$, $A_0 := 2B$, $A_k := B(0, s_{k+1}) \setminus B(0, s_k)$ for $k \ge 1$, and $b_k := \varepsilon_k = \eta(s_k)\, s_k^{\,n/p}$ for $k \ge 1$. By monotonicity, $\sum_{k \ge 1} b_k^{\,p} \le C \int_{r}^\infty \bigl( \eta(s)\, s^{n/p} \bigr)^p\, \tfrac{ds}{s}$, and this integral is at most $C_\eta$ over $[1, \infty)$ by~\eqref{eq:majorant-conditions} and at most $\int_0^1 s^{n - 1}\, ds = 1/n$ over $[\min(r, 1), 1]$ since $\eta \le 1$, so that $\sum_{k \ge 1} b_k^{\,p} \le C'_\eta$, uniformly in $r > 0$. Throughout, Hölder's Inequality on an annulus turns \textup{(i)} into the moment bound $\int_{A_j} \abs{m(y)}\, \abs{y}^{\abs{\nu}}\, dy \le \norm{m}_{L^q(A_j)}\, \norm{\abs{y}^{\abs{\nu}}}_{L^{q'}(A_j)} \le C\, \eta(s_j)\, s_j^{\,\abs{\nu} + n}$ for $j \ge 1$, the powers of $s_j$ collecting through $n/p - n = n_p$, and into $\int_{A_0} \abs{m(y)}\, \abs{y}^{\abs{\nu}}\, dy \le C\, r^{\abs{\nu} - n_p}$ on $2B$. On the reference ball $B(0, 2)$ and annulus $B(0, 2) \setminus B(0, 1)$, the Gram matrix of the monomials $\{ y^\mu \}_{\abs{\mu} \le N_p}$ is invertible, and its dilated dual polynomials give $\pi^\nu_k$ supported in $A_k$ with $\int_{A_k} \pi^\nu_k(y)\, y^\mu\, dy = \delta_{\nu \mu}$ and $\norm{\pi^\nu_k}_\infty \le C\, s_k^{-n - \abs{\nu}}$. With the tail moments $N^\nu_k := \int_{\abs{y} \ge s_k} m(y)\, y^\nu\, dy$ for $k \ge 1$ and $N^\nu_0 := \int_{\R^n} m(y)\, y^\nu\, dy$, absolutely convergent, with $N^\nu_0 = 0$ by \textup{(ii)} of the definition of these molecules, so that $\int_{A_k} m(y)\, y^\nu\, dy = N^\nu_k - N^\nu_{k+1}$ for every $k \ge 0$ and $g_k := m \one_{A_k} - \sum_{\abs{\nu} \le N_p} \bigl( \int_{A_k} m\, y^\nu \bigr) \pi^\nu_k$, Abel summation gives:
\[
  m \;=\; \sum_{k \ge 0} g_k
  \;+\; \sum_{k \ge 1} \sum_{\abs{\nu} \le N_p}
     N^\nu_k \bigl( \pi^\nu_k - \pi^\nu_{k-1} \bigr),
\]
convergent absolutely a.e.\ and in $L^1$, every summand supported in $2^{k+1} B$ with vanishing moments up to $N_p$. For the corrected pieces, \textup{(i)} on $A_k$ and the moment bound above, against $\norm{\pi^\nu_k}_q \le C\, s_k^{-n - \abs{\nu} + n/q}$, give $\norm{g_k}_q \le C\, \varepsilon_k\, \abs{2^{k+1} B}^{1/q - 1/p}$ for every $k \ge 0$, the powers of $s_k$ cancelling identically, so that $g_k = \lambda_k b'_k$ with $b'_k$ an $L^q$ atom associated with $2^{k+1} B$ and $\lambda_k \le C\, \varepsilon_k$. Hence $\lambda_0 \le C$ and $\sum_{k \ge 1} \lambda_k^p \le C \sum_{k \ge 1} b_k^{\,p} \le C\, C'_\eta$, uniformly in $r > 0$. For the tail-moment stream, each difference $\pi^\nu_k - \pi^\nu_{k-1}$ is $C\, s_k^{-n - \abs{\nu}}$ times an $L^\infty$ atom associated with $2^{k+1} B$, so its coefficient is $\lambda^\nu_k \le C \abs{N^\nu_k}\, s_k^{\,n_p - \abs{\nu}}$, and decomposing the tail moment over the annuli $A_j$, $j \ge k$, by the moment bound:
\[
  \lambda^\nu_k
  \;\le\; C\, s_k^{\,n_p - \abs{\nu}} \sum_{j \ge k}
     \eta(s_j)\, s_j^{\,\abs{\nu} + n}
  \;=\; C \sum_{j \ge k}
     b_j\, 2^{-(j - k)(n_p - \abs{\nu})} .
\]
For $\abs{\nu} < n_p$ this is a discrete convolution of $(b_j) \in \ell^p$ with a geometric sequence, hence in $\ell^p$ by $p$-sub-additivity, with $\ell^p$-mass at most $C\, C'_\eta$, uniformly in $r > 0$. For $\abs{\nu} = N_p = n_p$ (the integer case), the geometric factor is absent, and $\sum_{j \ge k} b_j \le 2^{n/p} E_\eta(s_{k-1})$ by monotonicity, so that, again by monotonicity, the scales $s_k \ge 1$ contribute $\sum_{k \ge 1,\, s_k \ge 1} (\lambda^\nu_k)^p \le C \int_{1/4}^\infty E_\eta(t)^p\, \tfrac{dt}{t} \le C \bigl( E_\eta(1/4)^p + C_\eta \bigr)$, controlled by the integer condition of~\eqref{eq:majorant-conditions} and by $E_\eta(1/4) \le 1 + E_\eta(1) < \infty$, uniformly in $r > 0$. For $r \ge 1$ this exhausts the stream. For $r < 1$ the scales $s_k < 1$, at most $\log_2(1/r)$ of them, remain, and there the moment clause \textup{(iii)} intervenes: by \textup{(ii)}, $N^\nu_1 = - \int_{2B} m(y)\, y^\nu\, dy$, so that $\abs{N^\nu_1} \le (1 + \log(1/r))^{-1/p}$, and $N^\nu_k = N^\nu_1 - \sum_{1 \le j < k} \int_{A_j} m\, y^\nu$ with $\abs[\big]{\int_{A_j} m\, y^\nu} \le C\, \eta(s_j)\, s_j^{\,n/p} \le C\, s_j^{\,n/p}$, since $\abs{\nu} + n = n/p$ and $\eta \le 1$, whence $\lambda^\nu_k \le C \bigl( (1 + \log(1/r))^{-1/p} + s_k^{\,n/p} \bigr)$ for $s_k < 1$. Summing the $p$-th powers over these scales, the first term contributes at most $C \log_2(1/r) / (1 + \log(1/r)) \le C$ and the second at most $C \sum_{s_k < 1} s_k^{\,n} \le C$, so this part of the stream is uniformly bounded as well. Without \textup{(iii)}, the clause $k = 0$ and Hölder's Inequality on $2B$ give only $\abs{N^\nu_1} \le C$, and the same count yields $\norm{m}_{\Hp, \mathrm{at}} \le C (1 + \log(1/r))^{1/p}$, which is the growth the moment clause is designed to remove, and which Proposition~\ref{prop:log-sharp} shows to be attained. Altogether, the decomposition exhibits $m$ as an atomic series with $\ell^p$-coefficient mass at most $C(n, p, q, \eta)^p$, which is the claimed bound on $\norm{m}_{\Hp, \mathrm{at}}$.
\end{proof}

\begin{proposition}
\label{prop:log-sharp} Let $0 < p \le 1$ with $n_p \in \Z$, let $1 < q \le \infty$ and $\eta \in \mathcal D_p$. Fix a multi-index $\nu_0$ with $\abs{\nu_0} = n_p$ and $\Psi \in C^\infty_c(B(0, 1))$ with $\int \Psi(y)\, y^\mu\, dy = \delta_{\mu \nu_0}$ for all $\abs{\mu} \le N_p$, and set $D_t := t^{-n/p}\, \Psi(\cdot / t)$ for $t > 0$. For $0 < r < 1$, the function $m_r := D_r - D_1$ satisfies the conditions \textup{(i)} and \textup{(ii)} of a $(p, q, \eta)$-molecule associated with $B(0, r)$, up to a constant depending only on $(n, p, q, \eta(1))$, with $\int_{B(0, 2r)} m_r(y)\, y^{\nu_0}\, dy = 1 + O(r^{n/p})$, and there are $r_0 \in (0, 1)$ and $c > 0$, depending only on $(n, p)$, such that $\norm{m_r}_{\Hp} \ge c\, (\log(1/r))^{1/p}$ for every $0 < r \le r_0$. Consequently, the bound $C (1 + \log(1/r_B))^{1/p}$ available without the clause \textup{(iii)} is attained, and \textup{(iii)} cannot be relaxed to any rate $\theta(r_B)$ with $\theta(r_B)\, (1 + \log(1/r_B))^{1/p} \to \infty$ as $r_B \to 0$.
\end{proposition}

\begin{proof}
Fix $\Psi_0 \in C^\infty_c(B(0, 1))$, non-negative, positive on $B(0, 1/2)$, with $\int \Psi_0 = 1$. The Gram matrix $\bigl( \int y^{\mu + \lambda}\, \Psi_0(y)\, dy \bigr)_{\abs{\mu}, \abs{\lambda} \le N_p}$ is positive definite, a non-zero polynomial not vanishing on the open set where $\Psi_0 > 0$, so $\Psi := \Psi_0 \sum_{\abs{\mu} \le N_p} a_\mu\, y^\mu$, with $a$ the solution of the linear system with right side $(\delta_{\mu \nu_0})_\mu$, has the required moments, and $\Psi$ depends only on $(n, p)$. Dilation gives $\int D_t(y)\, y^\mu\, dy = t^{\abs{\mu} - n_p}\, \delta_{\mu \nu_0} = \delta_{\mu \nu_0}$ for $\abs{\mu} \le N_p$, since $n + n_p = n/p$, so \textup{(ii)} holds for $m_r$. Write $B := B(0, r)$ and $s_k := 2^k r$. On $2B$, $\norm{D_r}_q = \norm{\Psi}_q\, r^{n(1/q - 1/p)} = \omega_n^{1/p - 1/q} \norm{\Psi}_q\, \abs{B}^{1/q - 1/p}$, while $\norm{D_1}_{L^q(2B)} \le \norm{\Psi}_\infty\, \abs{2B}^{1/q} \le 2^{n/q} \omega_n^{1/p} \norm{\Psi}_\infty\, \abs{B}^{1/q - 1/p}$ as $r < 1$, which is the clause $k = 0$ of \textup{(i)} up to a constant, with the usual modification at $q = \infty$. On $A_k(B)$, $k \ge 1$, $D_r$ vanishes, so $m_r = - D_1$ there, which vanishes when $s_k \ge 1$, while for $s_k < 1$, $\norm{D_1}_{L^q(A_k(B))} \le \norm{\Psi}_\infty\, \abs{A_k(B)}^{1/q} = C\, s_k^{\,n/q} \le C\, \eta(1)^{-1}\, \eta(s_k)\, s_k^{\,n/q} = C\, \eta(1)^{-1}\, \omega_n^{1/p - 1/q}\, \varepsilon_k\, \abs{2^k B}^{1/q - 1/p}$ by monotonicity of $\eta$, which is \textup{(i)} up to a constant depending on $\eta(1)$. The critical moment on $2B$ is $\int_{2B} D_r\, y^{\nu_0} - \int_{2B} D_1\, y^{\nu_0} = 1 + O(\norm{\Psi}_\infty\, r^{n + n_p})$, with $n + n_p = n/p$.

For the lower bound, fix $\Phi \in C^\infty_c(\R^n)$ with $\int \Phi \ne 0$ and $\partial^{\nu_0} \Phi(\mathbf{e}_1) \ne 0$ at the first coordinate vector $\mathbf{e}_1$, for instance $\Phi := \Psi_0 + \delta\, (\cdot - \mathbf{e}_1)^{\nu_0}\, \Psi_0(\cdot - \mathbf{e}_1)$ with $\delta > 0$ small: Leibniz's rule gives $\partial^{\nu_0} \Phi(\mathbf{e}_1) = \delta\, \nu_0!\, \Psi_0(0) > 0$, since $\Psi_0$ vanishes near $\mathbf{e}_1$ and every derivative of $(\cdot - \mathbf{e}_1)^{\nu_0}$ of order below $n_p$ vanishes at $\mathbf{e}_1$, while $\int \Phi \ge 1 - \delta \abs{\int y^{\nu_0} \Psi_0} > 0$. By continuity, $\abs{\partial^{\nu_0} \Phi(\omega)} \ge c_\Phi := \tfrac12 \abs{\partial^{\nu_0} \Phi(\mathbf{e}_1)}$ on a cap $\Omega \subseteq S^{n - 1}$ around $\mathbf{e}_1$, of positive surface measure $\sigma(\Omega)$. We compute $\norm{\cdot}_{\Hp, \max}$ with this $\Phi$, all such choices being equivalent (Subsection~\ref{ssec:HP}). Fix $x \ne 0$, put $t := \abs{x}$ and $\omega := x / \abs{x}$, and suppose $\abs{x} \ge K r$ with $K \ge 2$ to be chosen. Taylor's formula for $\Phi_t$ at $x$, to order $n_p$, with $\norm{\partial^\mu \Phi_t}_\infty \le C_\Phi\, t^{-n - \abs{\mu}}$ for $\abs{\mu} \le n_p + 1$, reads $\Phi_t(x - y) = \sum_{\abs{\mu} \le n_p} \frac{(-y)^\mu}{\mu!}\, \partial^\mu \Phi_t(x) + R_t(x, y)$ with $\abs{R_t(x, y)} \le C_\Phi\, \abs{y}^{n_p + 1}\, t^{-n - n_p - 1}$, so that, by the moments of $D_r$ and $\int \abs{y}^{n_p + 1} \abs{D_r(y)}\, dy \le \norm{\Psi}_1\, r$:
\[
  \Phi_t * D_r(x)
  \;=\; \frac{(-1)^{n_p}}{\nu_0!}\,
     \partial^{\nu_0} \Phi_t(x)
  \;+\; O\Bigl( C_\Phi \norm{\Psi}_1\,
     \frac{r}{\abs{x}}\, \abs{x}^{-n/p} \Bigr),
  \qquad
  \partial^{\nu_0} \Phi_t(x)
  \;=\; \abs{x}^{-n/p}\, (\partial^{\nu_0} \Phi)(\omega),
\]
while $\abs{\Phi_t * D_1(x)} \le \norm{\Phi}_1 \norm{\Psi}_\infty =: C_1$. Choose:
\[
  K := \frac{4\, \nu_0!\, C_\Phi \norm{\Psi}_1}{c_\Phi},
  \qquad
  \delta_0 := \min\Bigl( \frac12,\;
     \Bigl( \frac{c_\Phi}{4\, \nu_0!\, C_1} \Bigr)^{p/n} \Bigr).
\]
For $\omega \in \Omega$ and $K r \le \abs{x} \le \delta_0$, the three bounds give:
\[
  \sup_{t > 0} \abs{\Phi_t * m_r(x)}
  \;\ge\; \abs{\Phi_{\abs{x}} * m_r(x)}
  \;\ge\; \frac{c_\Phi}{4\, \nu_0!}\, \abs{x}^{-n/p} .
\]
Integrating $\abs{x}^{-n}$ in polar coordinates over this region, of angular measure $\sigma(\Omega)$, gives $\norm{m_r}_{\Hp, \max}^p \ge \bigl( \tfrac{c_\Phi}{4 \nu_0!} \bigr)^p \sigma(\Omega)\, \log\bigl( \delta_0 / (K r) \bigr) \ge \tfrac12 \bigl( \tfrac{c_\Phi}{4 \nu_0!} \bigr)^p \sigma(\Omega)\, \log(1/r)$ for $r \le r_0 := (\delta_0 / K)^2$, which is the claim, the constants depending only on $(n, p)$ through $\Psi$ and $\Phi$. The final assertions follow by scaling: without \textup{(iii)}, $m_r / C$ is a molecule with $\norm{m_r / C}_{\Hp} \ge (c / C)\, (\log(1/r))^{1/p}$, and under a relaxed rate $\theta$, the multiple $\theta(r)\, m_r / C'$, whose critical moment on $2B$ is at most $\theta(r)$, has quasi-norm at least $(c / C')\, \theta(r)\, (\log(1/r))^{1/p}$.
\end{proof}

\section{Goldberg's splitting: sufficiency and sharpness of the endpoints}
\label{sec:sharpness}
\label{ssec:endpoint}
In this section, we prove that the class~\eqref{eq:admiss-p} is exactly calibrated on all three of its axes, so that $\operatorname{Admiss}_{p}$ is maximal for the Goldberg-type splitting (Theorem~\ref{thm:admiss-characterization}). The three axes differ in kind, not merely in name: cancellation is an exact algebraic requirement, admitting no partial credit or rate, while regularity and decay are quantitative, each admitting its own sharp threshold, and the necessity of each is accordingly established by a separate construction, tailored to that kind of failure, rather than by one argument specialized three ways. Sufficiency is proved directly within Theorem~\ref{thm:admiss-characterization}, and the necessity of each condition is established by a construction on its axis, producing in every case functions admissible in every respect except the one under scrutiny. On the regularity axis, a lacunary counterexample defeats every $\omega \notin \mathcal R_p$, the uniform bounds on mollified atoms failing at a quantified rate and Goldberg's splitting failing at a single explicit $f \in \hp$ (Proposition~\ref{prop:linftyc-fails}). On the decay axis, a single atom establishes the necessity of~\eqref{eq:majorant-conditions}, the failing clause of Goldberg's splitting sorted by the failing integral: a divergent fractional integral defeats the $L^p$-clause, while at integer $n_p$ the fractional integral may converge while the iterated one diverges, in which case the $L^p$-clause holds (Lemma~\ref{lem:conv-atom}\textup{(iii)} consuming only the fractional condition) and the failure migrates to the $\Hp$-clause through the top order tail moments (Proposition~\ref{prop:decay-endpoint-fails}). On the cancellation axis, for each moment condition, a single explicit $f \in \hp$, independent of the pair $(\omega, \eta)$, witnesses Goldberg's splitting failing for every $\varphi$ violating that condition alone (Proposition~\ref{prop:cond-g-sharp}). The three sharpness propositions below then assemble, with the sufficiency argument, into the characterization: Goldberg's splitting selects exactly the admissible pairs and~\eqref{eq:admiss-p} is the largest class the Goldberg splitting tolerates.

\begin{proposition}
\label{prop:linftyc-fails} Let $0 < p < 1$, let $0 \le k \le N_p'$, and let $\omega$ be a modulus with $\displaystyle \sup_{0 < t \le 1}\, \omega(t)\, t^{-(n_p - k)} = \infty$. There exist $\varphi \in C^{k, \omega}_c(\R^n)$ satisfying $\operatorname{Cond}_G$, so that $\varphi$ fails~\eqref{eq:admiss-p} only through its regularity, being dominated together with all its derivatives by every majorant on account of its compact support, a sequence $(a_j)$ of $L^\infty$ $\hp$-atoms on balls $B_j = B(0, \sqrt{n}\, 2^{-j})$ with vanishing moments up to $N_p$, and a single $f \in \hpn \cap L^1_{\mathrm{loc}}(\R^n)$, such that, along a sequence of scales $j = i_m \to \infty$:
\[
  \norm{\varphi * a_{i_m}}_p
  \;\ge\; c_*\, \Omega_k(2^{-i_m})
  \;\xrightarrow[m \to \infty]{}\; \infty,
  \qquad
  \norm{a_{i_m} - \varphi * a_{i_m}}_{\Hpn}
  \;\xrightarrow[m \to \infty]{}\; \infty ,
\]
where $\Omega_k(t) := \omega(t)\, t^{-(n_p - k)}$ and $c_* > 0$ depends only on $(n, p, k)$, and the pointwise defined convolution $\varphi * f$ does not belong to $\Lpn$. In particular, the uniform bounds of Lemma~\ref{lem:goldberg-extension}\textup{(ii)}, Lemma~\ref{lem:conv-atom}\textup{(iii)} and~\eqref{eq:unif-Linfty} fail for this $\varphi$, and with them not only the term-wise treatment of Goldberg's splitting but its conclusion itself, at the single $f$. At $k = N_p'$ this is the necessity of the modulus condition~\eqref{eq:modulus-condition} in Theorem~\ref{thm:admiss-characterization}, and on the slice $\omega(t) = t^\theta$, it recovers the failure of every Hölder total order $k + \theta$ strictly below $n_p$.
\end{proposition}

\begin{proof}
Write $x = (x_1, x') \in \R \times \R^{n-1}$, $n = 1$ leaving the primed factors absent, and assume $\omega(t)/t$ unbounded as $t \to 0$, the remaining moduli being covered by the closing reduction. Fix $h \in C^\infty_c(\R)$ with $\supp(h) \subseteq [-1, 1]$, $h \ge 0$, $h \not\equiv 0$, set $g := h^{(N_p + 1)}$, fix a bump $0 \le h' \in C^\infty_c(\R^{n-1})$ with $\supp(h') \subseteq \overline{B'(0, 1)}$, $h' \not\equiv 0$, and define:
\[
  a_j(y) := c_0\, 2^{\,j n / p}\, g(2^j y_1)\, h'(2^j y'),
  \qquad j \ge 1,
\]
with $c_0$ normalizing $\norm{a_j}_\infty = \abs{B_j}^{-1/p}$. The support satisfies $\supp(a_j) \subseteq [-2^{-j}, 2^{-j}]^n \subseteq \overline{B_j}$, and the moments factorize: for $\abs{\nu} \le N_p$, the factor $\int y_1^{\nu_1}\, g(2^j y_1)\, dy_1$ vanishes by integration by parts, since $\nu_1 \le N_p$. Each $a_j$ is thus an $L^\infty$ $\hp$-atom on $B_j$, no moment conditions being asked of $h'$. With $\hat g(\xi) = (\mathrm{i} \xi)^{N_p + 1}\, \hat h(\xi)$ (one-dimensional transform), we record $\varrho_g := \abs{\hat g(1)} \ge \int h(u) \cos u\, du > 0$ (as $\supp(h) \subseteq [-1, 1] \subset (-\pi/2, \pi/2)$), together with $\abs{\hat g(\xi)} \le \norm{h}_1 \abs{\xi}^{N_p + 1}$ and $\abs{\hat g(\xi)} \le \norm{h^{(N_p + 2)}}_1 \abs{\xi}^{-1}$.

Choose scales $i_1 < i_2 < \dots$ inductively, each large enough that:
\[
  \Omega_k(2^{-i_m}) \ge m,
  \qquad
  \omega(2^{-i_m})\, 2^{i_m}
  \ge 2\, \omega(2^{-i_{m-1}})\, 2^{i_{m-1}},
  \qquad
  \omega(2^{-i_m}) \le \tfrac12\, \omega(2^{-i_{m-1}}) .
\]
The first is possible by the hypothesis and dyadic rounding, $\Omega_k$ changing by at most the factor $2^{n_p - k}$ between consecutive dyadic points by the monotonicity of $\omega$, the second because the slopes $\omega(2^{-i})\, 2^{i}$ are non-decreasing, $\omega(t) / t$ being non-increasing, and unbounded by assumption, and the third because $\omega(0^+) = 0$. Since the $i_m$ are strictly increasing integers, the frequencies $2^{i_m}$ form a lacunary sequence with gap ratio at least $2$, and we build a lacunary cosine series with these frequencies, the classical device for realizing a prescribed modulus of continuity~\cite[Chapter~V]{zygmund2002}. Fix $\chi \in C^\infty_c(\R)$ with $\chi \equiv 1$ on $[-\pi - 1, \pi + 1]$, $\supp(\chi) \subseteq [-\pi - 2, \pi + 2]$, and $\zeta \in C^\infty_c(\R^{n-1})$ with $\zeta \ge 0$, $\zeta \ge 1$ on $\overline{B'(0, 2)}$, and set:
\[
  \varphi_0(x) := \chi(x_1)\, W(x_1)\, \zeta(x'),
  \qquad
  W(t) := \sum_{m \ge 1} c_m \cos(2^{i_m} t),
  \qquad
  c_m := \omega(2^{-i_m})\, 2^{-i_m k} .
\]
Term-wise differentiation $l \le k$ times gives coefficients $\omega(2^{-i_m})\, 2^{-i_m(k-l)} \le \omega(2^{-i_m})$, geometrically summable by the third constraint, so $W \in C^k(\R)$. For the top derivative $W^{(k)}(t) = \sum_{m \ge 1} \omega(2^{-i_m}) \cos(2^{i_m} t + k\pi/2)$, split at $m(h) := \max\{m : 2^{-i_m} \ge \abs{h}\}$ for $0 < \abs{h} \le 1$ and bound each increment by $\min(2, 2^{i_m}\abs{h})$: the low terms $m \le m(h)$ sum geometrically by the second constraint to at most $2\,\omega(2^{-i_{m(h)}}) 2^{i_{m(h)}}\abs{h} \le 2\,\omega(\abs{h})$, using $\omega(t)/t$ non-increasing and $2^{-i_{m(h)}} \ge \abs{h}$, and the high terms $m > m(h)$ sum geometrically by the third constraint to at most $4\,\omega(2^{-i_{m(h)+1}}) \le 4\,\omega(\abs{h})$, using $2^{-i_{m(h)+1}} < \abs{h}$. Hence $\chi W \in C^{k,\omega}_c(\R)$, and tensoring with $\zeta$, mixed derivatives factoring, $\varphi_0 \in C^{k,\omega}_c(\R^n)$. By Lemma~\ref{lem:moment-correction}, $\varphi := \varphi_0 + v$ with $v \in C^\infty_c(\overline{B(0,1)})$ prescribing $\int \varphi = 1$, $\int x^\nu \varphi\, dx = 0$ for $0 < \abs{\nu} \le N_p$, gives $\varphi \in C^{k,\omega}_c(\R^n)$ satisfying $\operatorname{Cond}_G$.

On $\Omega := [-\pi, \pi] \times \overline{B'(0, \tfrac12)}$, substituting $u = 2^j y_1$, $u' = 2^j y'$ at $j = i_M$, using $2^{jn/p} \cdot 2^{-jn} = 2^{jn_p}$ and $\chi \equiv 1$ on $[-\pi,\pi]+[-1,1]$:
\[
  (\varphi_0 * a_{i_M})(x)
  \;=\; c_0\, 2^{\,i_M n_p}\, I_M(x_1)\, J_M(x'),
  \qquad
  I_M(x_1) = \sum_{m \ge 1} c_m\,
     \Re\Bigl( e^{\mathrm{i} 2^{i_m} x_1}\,
     \hat g\bigl( 2^{i_m - i_M} \bigr) \Bigr),
\]
where $J_M(x') := \int \zeta(x' - 2^{-i_M} u')\, h'(u')\, du' \in \bigl[ \norm{h'}_1,\, \norm{\zeta}_\infty \norm{h'}_1 \bigr]$ for $\abs{x'} \le \tfrac12$, the argument of $\zeta$ staying in $\overline{B'(0, 2)}$. Since the $2^{i_m}$ are distinct positive integers and $\{\cos(m\cdot), \sin(m\cdot)\}_{m \ge 1}$ are orthogonal in $L^2(-\pi,\pi)$ with norms $\sqrt\pi$, Parseval's identity gives $\norm{I_M}_{L^2(-\pi,\pi)}^2 \ge \pi\, \varrho_g^2\, c_M^2$, the term $m = M$. In the other direction, splitting $\sum_m c_m \abs{\hat g(2^{i_m - i_M})}$ at $m = M$ and applying the first recorded bound below and the second above, the terms $m \le M$ read $c_m\, 2^{(i_m - i_M)(N_p + 1)}\, \norm{h}_1$, a geometrically increasing sequence since consecutive ratios equal $\bigl[ \omega(2^{-i_{m+1}}) 2^{i_{m+1}} / \omega(2^{-i_m}) 2^{i_m} \bigr] \cdot 2^{(i_{m+1} - i_m)(N_p - k)} \ge 2$ by the second constraint and $k \le N_p$, while the terms $m > M$ read $c_m\, 2^{-(i_m - i_M)}\, \norm{h^{(N_p + 2)}}_1$, geometrically decreasing by the third constraint. Both sums are dominated by their term at $m = M$, so $\norm{I_M}_{L^\infty(-\pi, \pi)} \le C_1\, c_M$ with $C_1 = C_1(h)$. The interpolation inequality $\int \abs{F}^2 \le \norm{F}_\infty^{2 - p} \int \abs{F}^p$ on $(-\pi, \pi)$, Fubini on $\Omega$, and the identity $2^{\,i_M n_p}\, c_M = \omega(2^{-i_M})\, 2^{\,i_M (n_p - k)} = \Omega_k(2^{-i_M})$ then yield:
\begin{equation}
\label{eq:main-lower}
\norm{\varphi_0 * a_{i_M}}_{L^p(\Omega)}^p
\;\ge\; c_2^p\; \Omega_k(2^{-i_M})^p,
\qquad
c_2^p := \pi\, c_0^p\, \norm{h'}_1^p\,
   \abs{B'(0, \tfrac12)}\, \varrho_g^2\, C_1^{\,p - 2}
   > 0 .
\end{equation}
Since $v \in C^\infty_c \subseteq C^{\delta_p}_{\eta_M}$ for power majorants $\eta_M$, Lemma~\ref{lem:conv-atom}\textup{(iii)} gives $\norm{v * a_j}_p \le C_v$ uniformly in $j$, and the $p$-triangle inequality yields:
\[
\norm{\varphi * a_{i_M}}_p^p
\;\ge\; \norm{\varphi_0 * a_{i_M}}_{L^p(\Omega)}^p
   - \norm{v * a_{i_M}}_{L^p(\Omega)}^p
\;\ge\; c_2^p\, \Omega_k(2^{-i_M})^p - C_v^p
\;\longrightarrow\; \infty,
\]
proving the first assertion with any $c_* < c_2$, by the first constraint.

Each $a_j$ is an $L^\infty$ atom of $\Hpn$, so $\norm{a_j}_{\Hp} \le C_3(n,p)$ uniformly~\cite{latter1978}, \cite[Chapter~III]{garcia-cuerva-rubio1985}. If $\norm{a_{i_{M_l}} - \varphi * a_{i_{M_l}}}_{\Hp} \le K$ along a subsequence, then $\norm{\varphi * a_{i_{M_l}}}_{\Hp}^p \le C_3^p + K^p$, but $u_l := \varphi * a_{i_{M_l}}$ is continuous with compact support, so for $\Phi \in \mathcal S(\R^n)$ with $\int \Phi = 1$, dominated convergence gives $\Phi_t * u_l \to u_l$ pointwise as $t \downarrow 0$, whence $\abs{u_l} \le \sup_{t > 0} \abs{\Phi_t * u_l}$ and $\norm{u_l}_p \lesssim \norm{u_l}_{\Hp}$, bounding $\norm{\varphi * a_{i_{M_l}}}_{\Lp}$ uniformly and contradicting the divergence just proved. Hence $\norm{a_{i_M} - \varphi * a_{i_M}}_{\Hp} \to \infty$.

For the single distribution, fix $R$ with $\supp(\varphi) \subseteq B(0, R)$, set $L_0 := 2(\pi + R + 2)$, $x_m := (m L_0, 0, \dots, 0)$, and $f := \sum_{m \ge 1} \kappa_m\, a_{i_m}(\cdot - x_m)$, $\kappa_m := \Omega_k(2^{-i_m})^{-1} m^{-1/p}$: by the first constraint $\Omega_k(2^{-i_m})^{-1} \le m^{-1}$, so $\sum_m \kappa_m^p \le \sum_m m^{-1-p} < \infty$ and $f \in \hpn$ by~\eqref{eq:atom-hp-bound} and $p$-sub-additivity. The supports being pairwise disjoint and locally finite, the series also converges pointwise, $f \in L^1_{\mathrm{loc}}(\R^n)$, and $\varphi * f$ is defined everywhere as an absolutely convergent integral. By the choice of $L_0$, the supports of the $\varphi * a_{i_m}(\cdot - x_m)$ are pairwise disjoint, each meeting no window $x_{m'} + \Omega$ with $m' \ne m$ (separation in the first coordinate), and the windows are pairwise disjoint. Hence, $\varphi * f$ coincides on $x_m + \Omega$ with the single term $\kappa_m\, (\varphi * a_{i_m})(\cdot - x_m)$, and by translation invariance, \eqref{eq:main-lower}, and the uniform bound $C_v$ above:
\[
  \norm{\varphi * f}_p^p
  \;\ge\; \sum_{m \ge 1} \kappa_m^p
     \Bigl( c_2^p\, \Omega_k(2^{-i_m})^p - C_v^p
     \Bigr)
  \;\ge\; c_2^p \sum_{m \ge 1} m^{-1}
     \;-\; C_v^p \sum_{m \ge 1} \kappa_m^p
  \;=\; \infty .
\]
Hence $\varphi * f \notin \Lpn$.

Finally, if $\omega(t) \le Ct$, then $\sup_t \Omega_k(t) = \infty$ forces $n_p - k > 1$, hence $k + 1 \le N_p'$, and a witness of the proposition at order $k+1$ and modulus $t^\theta$, $0 < \theta < n_p - k - 1$, for which $\omega(t)/t$ is unbounded, lies in $C^{k+1,\theta}_c \subseteq C^{k,\omega}_c$ up to the constant $\omega(1)^{-1}$, its order-$k$ derivatives Lipschitz by compact support, and its failure of (G) is failure for $C^{k,\omega}_c$.
\end{proof}

\begin{proposition}
\label{prop:decay-endpoint-fails} Let $0 < p \le 1$ and let $\eta$ be a majorant with $\int_1^\infty \eta(s)\, s^{N_p + n - 1}\, ds < \infty$ and $\eta \notin \mathcal D_p$. There exist $\varphi \in C^\infty(\R^n) \cap C^{\delta_p}_{\eta}(\R^n)$, satisfying $\operatorname{Cond}_G$ and as such a member of every requirement of~\eqref{eq:admiss-p} except that its majorant falls outside $\mathcal D_p$, and a single large-ball $L^\infty$ $\hp$-atom $a$, such that Goldberg's splitting fails at $f = a$: through its $L^p$-clause, $\varphi * a \notin \Lpn$, when the fractional integral of~\eqref{eq:majorant-conditions} diverges, and through its $\Hp$-clause, $a - \varphi * a \notin \Hpn$, when the fractional integral converges while the iterated one diverges, which occurs only at integer $n_p$. This failure constitutes the necessity of the decay condition in Theorem~\ref{thm:admiss-characterization}, at both parities, and majorants with divergent moment integral admit no members of $\operatorname{Cond}_G$ with absolutely convergent normalization at all.
\end{proposition}

\begin{proof}
Since $\eta \notin \mathcal D_p$, either the fractional integral of~\eqref{eq:majorant-conditions} diverges, treated first, or it converges while the iterated one diverges, only possible at integer $n_p$. Fix $\vartheta \in C^\infty_c((1, 2))$, $\vartheta \ge 0$, $\int_1^2 \vartheta(s)\, \tfrac{ds}{s} = 1$, and $u(t) := \int_1^2 \eta(ts)\, \vartheta(s)\, \tfrac{ds}{s} = \int_t^{2t} \eta(\sigma)\, \vartheta(\sigma/t)\, \tfrac{d\sigma}{\sigma}$: smooth on $(0,\infty)$, every $t$-derivative falling on $\vartheta(\sigma/t)$, with $\eta(2t) \le u(t) \le \eta(t)$ and $\abs{u^{(k)}(t)} \le C_k\, \eta(t)\, t^{-k}$ by monotonicity on $[t,2t]$. In particular:
\begin{equation}
\label{eq:inherited-divergence}
  E_u(t) \;\ge\; 2^{-n/p}\, E_\eta(2t),
\end{equation}
so $u$ inherits whichever divergence $\eta$ carries. Fix a solid spherical harmonic $Y$ of degree $n_p$, normalized in $L^2(S^{n-1})$, a smooth $\varsigma$ vanishing on $B(0, 2)$ and $\equiv 1$ off $B(0, 4)$, and set:
\[
  \varphi := \varphi_{\mathrm{rad}}
     + \kappa\, \varphi_{\mathrm{ang}} + v,
  \qquad
  \varphi_{\mathrm{rad}}(y) := u(\abs{y})\, \varsigma(y),
  \qquad
  \varphi_{\mathrm{ang}}(y) := Y\bigl( y / \abs{y} \bigr)\,
     u(\abs{y})\, \varsigma(y),
\]
with $\kappa := 0$ in the first case and $\kappa \in \{1, 2\}$ chosen in the second, and $v \in C^\infty_c(\overline{B(0, 1)})$ prescribing, by Lemma~\ref{lem:moment-correction}, $\int \varphi = 1$ and $\int y^\nu \varphi\, dy = 0$ for $0 < \abs{\nu} \le N_p$ (for $n_p = 0$, set $\varphi_{\mathrm{ang}} := 0$). All moment integrals converge absolutely, the top order dominated by the moment hypothesis and the lower orders by the pointwise bound $\int_t^\infty \eta(s)\, s^{\,j + n - 1}\, ds \le t^{\,j - n_p}\, E_\eta(t)$ at $t = 1$, and each piece lies in $C^\infty \cap C^{\delta_p}_{\eta}$, the derivatives of $u(\abs{y})\, \varsigma$ and of the degree zero homogeneous $Y(y / \abs{y})$ being bounded by $C\, \eta(\abs{y})\, (1 + \abs{y})^{-\abs{\nu}}$ off $B(0, 2)$. Thus $\varphi$ satisfies all the stated conditions. Let $a := \abs{B(0, 1)}^{-1/p} \one_{B(0, 1)}$, of mass $m_a = \omega_n^{1 - 1/p}$.

In the first case, the failure lies in the $L^p$-clause: with $\kappa = 0$, $\varphi - v = \varphi_{\mathrm{rad}} \ge 0$, equal to $u(\abs{y})$ for $\abs{y} \ge 4$, and for $\abs{x} \ge 8$, $\abs{y} \le 1$, monotonicity gives $(\varphi_{\mathrm{rad}} * a)(x) \ge m_a\, u(\abs{x}+1) \ge m_a\, \eta(4\abs{x})$, while $v * a$ is supported in $B(0,2)$. In polar coordinates, after $\sigma = 4s$:
\[
  \int_{\R^n} \abs{\varphi * a}^p
  \;\ge\; m_a^p \int_{\abs{x} \ge 8}
     \eta(4 \abs{x})^p\, dx
  \;=\; c \int_{32}^\infty
     \bigl( \eta(\sigma)\, \sigma^{n/p} \bigr)^p\,
     \frac{d\sigma}{\sigma}
  \;=\; \infty ,
\]
by the assumed divergence of the fractional integral, a shift of the lower limit not affecting it. As $a \in \hpn$, the $L^p$-clause of Goldberg's splitting fails at $f = a$.

In the second case, the failure migrates to the $\Hp$-clause. The fractional integral converges, $h(s) := \eta(s)\, s^{n/p}$ satisfying $\int_1^\infty h(s)^p \tfrac{ds}{s} < \infty$, while $\int_1^\infty E_u(t)^p \tfrac{dt}{t} = \infty$ by~\eqref{eq:inherited-divergence} and $n_p \in \Z$, so $N_p = n_p$. Set $g := a - \varphi * a$, bounded with $\abs{g(x)} \le C\, \eta(\abs{x}/2)$ for $\abs{x} \ge 8$ by~\eqref{eq:sep-majorant} and $\int g\, y^\nu\, dy = 0$ for $\abs{\nu} \le n_p$ by the binomial identity.

The product structure makes the top order tail moments of $\varphi$ exact multiples of one radial function: for $t \ge 8$ and $\abs{\nu} = n_p$, $\varsigma \equiv 1$ there and $E_u(t) = \int_t^\infty u(s)\, s^{\,n_p+n-1}\, ds$, so:
\[
  \int_{\abs{w} \ge t} w^\nu\, \varphi(w)\, dw
  \;=\; c_\nu\, E_u(t),
  \qquad
  c_\nu := \int_{S^{n-1}} \omega^\nu\, d\sigma
     + \kappa \int_{S^{n-1}} \omega^\nu\, Y(\omega)\,
     d\sigma .
\]
The vector $\bigl( \int \omega^\nu Y\, d\sigma \bigr)_{\abs{\nu} = n_p}$ is non-zero: $Y|_{S^{n-1}}$ is, up to normalization, the harmonic component of the degree-$n_p$ homogeneous polynomials~\cite[Chapter~IV]{steinweiss1971}, hence a combination of the $\omega^\mu$, $\abs{\mu} = n_p$, whose joint annihilation would force $\int Y^2\, d\sigma = 0$, so at most one $\kappa$ annihilates $(c_\nu)$: choose $\kappa \in \{1, 2\}$ with $(c_\nu) \ne 0$ and fix $\nu_0$ with $c_{\nu_0} \ne 0$ (for $n_p = 0$: $\nu_0 = 0$, $c_0 = \abs{S^{n-1}} > 0$). Transferring through the binomial expansion of $(w + z)^\nu$, $\abs{z} \le 1$, the boundary annulus $t - 1 \le \abs{w} \le t + 1$ contributing at most $C\, h(t / 2)$ and the lower binomial orders at most $C\, t^{-1} E_\eta(t)$ by the pointwise bound of Step 0:
\[
  N^\nu(t) := \int_{\abs{y} \ge t} y^\nu\, g(y)\, dy
  \;=\; -\, m_a\, c_\nu\, E_u(t)
  \;+\; O\bigl( h(t / 2) + t^{-1} E_\eta(t) \bigr),
  \qquad \abs{\nu} = n_p .
\]

Set $A(t) := t^{-1} \int_4^t h(s)\, ds$ and, for $0 \le j < n_p$, $P_j(t) := t^{\,n_p-j} \int_t^\infty \eta(s)\, s^{\,j+n-1}\, ds$: then
\begin{equation}
\label{eq:error-statistics}
\int_1^\infty h(t / 2)^p\, \frac{dt}{t}
\;+\; \int_1^\infty A(t)^p\, \frac{dt}{t}
\;+\; \sum_{0 \le j < n_p}
   \int_1^\infty P_j(t)^p\, \frac{dt}{t}
\;<\; \infty :
\end{equation}
the first integral is the fractional condition, and by monotonicity $A(2^k) \le C \sum_{l \le k} h(2^{l-1})\, 2^{\,l - k}$ and $P_j(2^k) \le C \sum_{l \ge k} h(2^l)\, 2^{-(l - k)(n_p - j)}$, discrete convolutions of the $\ell^p$ sequence $\bigl( h(2^l) \bigr)$ with geometric sequences, hence in $\ell^p$ by $p$-sub-additivity, the dyadic sums comparable to the integrals.

The functionals $\Phi \mapsto \int \Phi$, $\Phi \mapsto \partial^\nu \Phi(0)$ ($\abs{\nu} = n_p$) are linearly independent, hence jointly surjective on $\mathcal S(\R^n)$: fix $\Phi$ with $\int \Phi = 1$, $\partial^{\nu_0} \Phi(0) = 1$, $\partial^\nu \Phi(0) = 0$ for the other $\abs{\nu} = n_p$. Suppose $g \in \Hp$, so that $\norm[\big]{\sup_{t > 0} \abs{\Phi_t * g}}_p \lesssim \norm{g}_{\Hp} < \infty$ by the maximal characterization with a single test function~\cite[Theorem~11]{feffermanstein1972}. For $t \ge C_0$, with $C_0 \ge 8$ a threshold fixed at the end of the proof, and $\abs{x} \le c_0 t$, $c_0 \in (0, 1)$, subtracting the degree $n_p$ Taylor polynomial of $y \mapsto \Phi_t(x - y)$ at $y = 0$ against the vanishing moments of $g$:
\[
  (\Phi_t * g)(x)
  = \int_{\abs{y} \le t} R_{t, x}\, g\, dy
  + \int_{\abs{y} > t} \Phi_t(x - y)\, g\, dy
  - \sum_{\abs{\nu} \le n_p}
     \frac{(-1)^{\abs{\nu}}}{\nu!}\,
     \partial^\nu \Phi_t(x)\, N^\nu(t) .
\]
The remainder is at most $C\, t^{-n - n_p} \bigl( t^{-1} + A(t) \bigr)$, via the bound $\abs{R_{t, x}(y)} \le C_\Phi\, t^{-n - n_p - 1} \abs{y}^{n_p + 1}$ and the identity $s^{\,n_p + n}\, \eta(s) = h(s)$, which collapses the integrand of $\int_{\abs{y} \le t} \abs{y}^{n_p + 1} \abs{g}\, dy$. The far field is at most $\norm{\Phi}_1 \sup_{\abs{y} > t} \abs{g} \le C\, t^{-n - n_p}\, h(t / 2)$. The lower orders $\abs{\nu} < n_p$ carry $\abs{N^\nu(t)} \le C\, t^{\,\abs{\nu} - n_p}\, P_{\abs{\nu}}(t)$ against the bound $\abs{\partial^\nu \Phi_t} \le t^{-n - \abs{\nu}} \norm{\partial^\nu \Phi}_\infty$, totalling $C\, t^{-n - n_p} \sum_j P_j(t)$. At top order, $\partial^\nu \Phi_t(x) = t^{-n - n_p} (\partial^\nu \Phi)(x / t)$, and the prescribed derivatives with the mean value theorem on $\abs{x / t} \le c_0$ deviate from their values at $0$ by at most $\norm{\nabla \partial^\nu \Phi}_\infty\, c_0$. Set $\varkappa := m_a \abs{c_{\nu_0}} / \nu_0!$ and:
\[
  G \;:=\; \Bigl\{ t \ge C_0 \;:\;
     h(t / 2) + A(t) + \sum_{0 \le j < n_p} P_j(t)
     \;\le\; \epsilon_0\, E_u(t) \Bigr\} .
\]
Choosing $c_0$ so small that the off $\nu_0$ top order contributions cost at most $\tfrac14 \varkappa$, then $\epsilon_0$ so small and $C_0$ so large that on $G$ the remainder, far field, lower orders, and the errors of the moment transfer total at most $\tfrac14 \varkappa\, t^{-n - n_p} E_u(t)$, we get $\abs{(\Phi_t * g)(x)} \ge \tfrac14 \varkappa\, t^{-n - n_p} E_u(t)$ for $t \in G$, $\abs{x} \le c_0 t$. Taking $t = 2\abs{x} / c_0$ over the annuli with good scales, and using $(n + n_p)\, p = n$:
\[
  \norm[\Big]{ \sup_{t > 0} \abs{\Phi_t * g} }_p^{\,p}
  \;\gtrsim\; \int_G E_u(t)^p\, \frac{dt}{t}
  \;=\; \int_{C_0}^\infty E_u(t)^p\, \frac{dt}{t}
     \;-\; \int_{G^c} E_u(t)^p\, \frac{dt}{t}
  \;=\; \infty ,
\]
the first integral divergent and the second finite, since on $G^c$ at least one statistic exceeds $\epsilon_0 E_u(t)$ and~\eqref{eq:error-statistics} bounds it by $\epsilon_0^{-p}$ times a finite quantity. This divergence contradicts $g \in \Hp$, so $a - \varphi * a \notin \Hpn$. Concluding the proof.
\end{proof}

\begin{proposition}
\label{prop:cond-g-sharp}
Let $0 < p \le 1$, let $\omega$ be a modulus and $\eta$ a majorant with $\int_1^\infty \eta(s)\, s^{N_p+n-1}\, ds < \infty$, and fix a multi-index $\alpha$ with $0 \le \abs\alpha \le N_p$. Let $\varphi \in C^{\omega}_{\eta}(\R^n)$ (respectively $\varphi \in L^\infty_{\eta}(\R^n)$ at $p = 1$, where $N_1 = 0$ and only $\alpha = 0$ arises) satisfy every condition of $\operatorname{Cond}_G$ except at order $\alpha$: that is, $\int \varphi \, dx \ne 1$ if $\alpha = 0$, or $\int x^\alpha \varphi(x) \, dx \ne 0$ if $1 \le \abs\alpha \le N_p$, with every other condition of $\operatorname{Cond}_G$ holding. Then there is $f \in \mathcal S(\R^n) \subset \hpn$, independent of $\varphi$, of $\alpha$, and of $(\omega, \eta)$, at which Goldberg's splitting fails.
\end{proposition}
\begin{proof}
The decay hypothesis on $\eta$ gives $x^\nu \varphi \in L^1(\R^n)$ for every $\abs\nu \le N_p$, in either case $C^\omega_\eta$ or $L^\infty_\eta$, so all moments of $\varphi$ up to order $N_p$ are finite and $\widehat\varphi \in C^{N_p}$ near $0$. Fix $\chi \in C_c^\infty(\R^n)$ with $\chi \equiv 1$ on a neighborhood of the origin and set $f := \check\chi$, so that $f \in \mathcal S(\R^n) \subset \hpn$ and $\widehat f \equiv 1$ near $0$, giving $\partial^\nu \widehat f(0) = 0$ for every $\nu \ne 0$ and $\widehat f(0) = 1$. Write $g := f - \varphi * f$, so $\widehat g = (1 - \widehat\varphi)\widehat f$, and by Leibniz:
\[
\partial^\alpha \widehat g(0) = \sum_{\beta \le \alpha} \binom\alpha\beta \, \partial^\beta(1 - \widehat\varphi)(0) \, \partial^{\alpha - \beta} \widehat f(0) .
\]
Every term with $\beta \ne \alpha$ carries the factor $\partial^{\alpha-\beta}\widehat f(0)$ with $\alpha - \beta \ne 0$, hence vanishes regardless of $\varphi$, leaving the single surviving term at $\beta = \alpha$:
\[
\partial^\alpha \widehat g(0) = -\partial^\alpha \widehat\varphi(0) \ \text{ if } 1 \le \abs\alpha \le N_p,
\qquad\qquad
\widehat g(0) = 1 - \widehat\varphi(0) \ \text{ if } \alpha = 0 ,
\]
and this quantity is non-zero by hypothesis. Up to the non-zero constant $(-2\pi i)^{\abs\alpha}$, it equals $\int x^\alpha g(x) \, dx$, so $g$ fails a moment vanishing condition of order $\abs\alpha \le N_p$, and every $\Hpn$ element satisfies exactly these vanishing conditions, so $g \notin \Hpn$. This concludes the argument.
\end{proof}

\begin{theorem}
\label{thm:admiss-characterization} Let $0 < p \le 1$, let $\omega$ be a modulus and $\eta$ a majorant with $\int_1^\infty \eta(s)\, s^{N_p + n - 1}\, ds < \infty$, the moment integrals of $\operatorname{Cond}_G$ being otherwise divergent for every candidate comparable to $\eta$. Then Goldberg's splitting holds for every $\varphi \in C^{\omega}_{\eta}(\R^n) \cap \operatorname{Cond}_G$ (for $p = 1$, every $\varphi \in L^\infty_{\eta}(\R^n) \cap \operatorname{Cond}_G$) if and only if $\omega \in \mathcal R_p$ and $\eta \in \mathcal D_p$.
\end{theorem}

\begin{proof}
For the \textup{if} direction, $\omega \in \mathcal R_p$ embeds $C^{\omega}_{\eta}$ into the extremal slice $C^{\delta_p}_{\eta}$, with $\norm{\varphi}_{\delta_p, \eta} \le C_\omega \norm{\varphi}_{\omega, \eta}$ as in Subsection~\ref{ssec:holder}, and $\eta \in \mathcal D_p$ then places $\varphi$ in $\operatorname{Admiss}_{p}$, with membership witnessed by $\eta$. Let $f = \sum_j \lambda_j\, a_j$ be Goldberg's atomic decomposition~\cite{goldberg1979}, with $C_6 := (\sum_j \abs{\lambda_j}^p)^{1/p}$, split by the size of the ball and relabeled as $f = \sum_k \lambda_{1,k} A_k + \sum_j \lambda_{2,j}\, a_j$, with large-ball atoms $A_k$ and small-ball atoms $a_j$ carrying $N_p$ vanishing moments. By Lemma~\ref{lem:termwise-convolution}, the series $\sum_k \lambda_{1,k} (\varphi * A_k) + \sum_j \lambda_{2,j} (\varphi * a_j)$ converges absolutely, uniformly, and in $\mathcal S'(\R^n)$, its sum being by definition $\varphi * f$, whence in $\mathcal S'(\R^n)$:
\[
  f - \varphi * f
  \;=\; \sum_k \lambda_{1,k}\, \bigl( A_k - \varphi * A_k
     \bigr)
  \;+\; \sum_j \lambda_{2,j}\, a_j
  \;-\; \sum_j \lambda_{2,j}\, (\varphi * a_j) .
\]
Each series is an $\ell^p$-combination of uniformly bounded elements of $\Hp$: the corrections and mollifications are, up to uniform constants and with $N_p$ vanishing moments, $(p, q, \eta_{1/2})$-molecules (Lemma~\ref{lem:goldberg-extension}\textup{(i)} and \textup{(ii)}, the latter at the exponent $q$ among its simultaneous conclusions), with $\eta_{1/2} \in \mathcal D_p$, and each $a_j$ is precisely an $\Hp$-atom. By the molecular embedding Theorem~\ref{lem:eta-molecule} at $\eta_{1/2}$, the uniform atomic bound~\eqref{eq:atom-hp-bound} in its $\Hp$ form, the equivalence of the atomic and maximal quasi-norms~\cite{latter1978}, \cite[Chapter~III]{garcia-cuerva-rubio1985}, and $p$-sub-additivity, each series converges in $\Hp$ with quasi-norm controlled by $C_6^p$, so $f - \varphi * f \in \Hp$ with $\norm{f - \varphi * f}_{\Hp} \le C(n, p, q, \eta, \norm{\varphi}_{\delta_p, \eta})\, C_6$. For $\varphi * f \in \Lpn$: by Lemma~\ref{lem:conv-atom}\textup{(iii)} and $p$-sub-additivity applied to the absolutely convergent series of Lemma~\ref{lem:termwise-convolution}, $\norm{\varphi * f}_p^p \le C^p \sum_{k} \abs{\lambda_{1,k}}^p + C^p \sum_j \abs{\lambda_{2,j}}^p \le C^p C_6^p$, with $C = C(n, p, \eta, \norm{\varphi}_{\delta_p, \eta})$ the uniform bound over all $L^q$ $\hp$-atoms.

For \textup{only if}, suppose first $\eta \notin \mathcal D_p$. Proposition~\ref{prop:decay-endpoint-fails} produces $\varphi \in C^\infty \cap C^{\delta_p}_{\eta} \cap \operatorname{Cond}_G$ and a single atom $a$ at which Goldberg's splitting fails, through the clause matching the divergent integral of~\eqref{eq:majorant-conditions}. Moreover $\varphi \in C^{\omega}_{\eta}$ for every modulus, its top order increments being bounded by $C \min(1, \abs{h}) \le C\, \omega(1)^{-1}\, \omega(\abs{h})$ through the monotonicity of $\omega(t) / t$, so Goldberg's splitting fails on the class. Suppose finally $\eta \in \mathcal D_p$ but $\omega \notin \mathcal R_p$. Proposition~\ref{prop:linftyc-fails} at $k = N_p'$ produces $\varphi \in C^{N_p', \omega}_c \cap \operatorname{Cond}_G$, contained in $C^{\omega}_{\eta}$ by compact support, and $f \in \hpn$ with $\varphi * f \notin \Lpn$, so Goldberg's splitting fails on the class. (For $p = 1$, only the decay direction is present.) Independently of $\omega$ and $\eta$, $\operatorname{Cond}_G$ itself is unweakenable: Proposition~\ref{prop:cond-g-sharp} furnishes one Schwartz $f \in \hpn$, the same $f$ for every order $\alpha$ with $0 \le \abs\alpha \le N_p$, at which Goldberg's splitting fails for any $\varphi$ meeting the decay and regularity side of an admissible pair but violating $\operatorname{Cond}_G$ at that order, so no relaxation of the normalization or of any one vanishing moment, exact, approximate, or merely finite, survives at any pair. The maximality clause follows: the admissible pairs are exactly $\mathcal R_p \times \mathcal D_p$ carrying $\operatorname{Cond}_G$ exactly as stated, and the union of their classes is~\eqref{eq:admiss-p} by the collapse onto the extremal slice recorded after~\eqref{eq:admiss-p}, every $C^{\omega}_{\eta}$ with $\omega \in \mathcal R_p$ lying inside $C^{\delta_p}_{\eta}$ with norm control while containing the slice itself.
\end{proof}

\section*{Acknowledgments}
The author wishes to express his deepest gratitude to Dr.~Galia Dafni, his master's thesis supervisor at Concordia University, whose guidance, insight, and encouragement shaped the thesis from which the present work grew, and whose feedback prompted the pursuit of both rigor and originality along the way. This work was partially supported by the CRM, Montréal, Canada.

\end{document}